\documentclass{amsart}
\usepackage{amssymb,amsmath,graphicx,mathtools}
\usepackage[latin1]{inputenc}
\usepackage[normalem]{ulem}
\usepackage{url}
\usepackage{tikz}
\usetikzlibrary{decorations.pathreplacing,calc,intersections,arrows}

\newtheorem{thm}{Theorem}[section]
\newtheorem{cor}[thm]{Corollary}
\newtheorem{lemma}[thm]{Lemma}

\newtheorem{prop}[thm]{Proposition}
\theoremstyle{definition}
\newtheorem{defn}[thm]{Definition}

\theoremstyle{remark}
\newtheorem{rem}[thm]{Remark}

\numberwithin{equation}{section}

\newcommand{\acts}{\ensuremath{\curvearrowright}}%
\newcommand{\sect}{\ensuremath{\dot}}%
\newcommand{\Z}{\mathbf{Z}}

\newcommand{\N}{\mathbf{N}}
\newcommand{\F}{\mathbf{F}}

\newcommand{\R}{\mathbf{R}}
\newcommand{\C}{\mathbf{C}}
\newcommand{\e}{\mathbf{e}}
\newcommand{\Q}{\mathbf{Q}}
\newcommand{\Pc}{\mathcal{P}_c}
\newcommand{\SL}{\ensuremath{\mathrm{SL}}}%
  \newcommand{\cO}{\ensuremath{\mathcal{O} }}%

\author[de la Salle]{Mikael de la Salle}
\address{UMPA, CNRS ENS de Lyon\\Lyon\\France}
\email{mikael.de.la.salle@ens-lyon.fr}
\thanks{M dlS was supported by ANR grants GAMME and AGIRA}
\title{Strong property (T) for higher rank lattices}
  
\begin{document}

\begin{abstract}

  We prove that every lattice in a product of higher rank simple Lie groups or higher rank simple algebraic groups over local fields has Vincent Lafforgue's strong property (T). Over non-archimedean local fields, we also prove that they have strong Banach proerty (T) with respect to all Banach spaces with nontrivial type, whereas in general we obtain such a result with additional hypotheses on the Banach spaces. The novelty is that we deal with non-cocompact lattices, such as $\mathrm{SL}_n(\Z)$ for $n \geq 3$. To do so, we introduce a stronger form of strong property (T) which allows us to deal with more general objects than group representations on Banach spaces that we call two-step representations, namely families indexed by a group of operators between different Banach spaces that we can compose only once. We prove that higher rank groups have this property and that this property passes to undistorted lattices.
\end{abstract}

\maketitle

\section{Introduction}

Kazhdan's property (T) is a rigidity property for unitary representations of a locally compact group, which has found numerous applications in various areas of pure and applied mathematics, see \cite{MR2415834}. Vincent Lafforgue's strong property (T) is a strengthening of property (T) which deals with representations by bounded operators with small exponential growth of the norm. Its introduction in \cite{MR2423763} was motivated by the Baum-Connes conjecture, as it is a natural obstruction to apply Lafforgue's approach to the Baum-Connes conjecture, see \cite{MR2732057}. It has also found several applications, notably its Banach-space version that we will discuss below, as it provided the first examples of superexpanders (expanders which do not coarsely embed into any uniformly convex Banach space), and as it implies strong fixed point properties for affine actions on Banach spaces. Another notable recent application is also to dynamics, as it was one of the steps in spectacular progresses on the Zimmer program \cite{brownFisherHurtado}.

So far strong property (T) has been shown for higher rank connected simple Lie groups (or higher rank simple algebraic groups over local fields) and their cocompact lattices. The case when the Lie algebra contains $\mathfrak{sl}_3$ was proven by Lafforgue in \cite{MR2423763}. The generalization to other algebraic groups was done by Liao \cite{liao} (for non-archimedean local fields) and de Laat and the author \cite{strongTsp4} (for archimedean local fields, \emph{i.e.} $\R$). In particular before the present work it was not known whether $\mathrm{SL}_3(\Z)$ has strong property (T). The aim of this article is to extend these results to cover the lattices which are not cocompact (for example $\mathrm{SL}_3(\Z)$) as well. This will have consequences on the Zimmer program \cite{brownFisherHurtado2}. We also take the opportunity to state and prove all the results more generally for (lattices in) semisimple groups rather than simple groups, and also to some non semisimple Lie groups (Remark \ref{rem:nonsemisimple}).

In the whole article, \emph{local field} will mean commutative, non-discrete locally compact topological field. So a local field is a finite extension of $\R$ (in which case it is \emph{archimedean}), or of $\Q_p$ or $\F_p(\!(t)\!)$ for some prime number $p$ (in which case it is \emph{non-archimedean}). \emph{Higher rank simple group} will mean either real connected simple Lie group of real rank $\geq 2$, or connected almost $\F$-simple algebraic group of $\F$-split rank $\geq 2$ over a local field $\F$. \emph{Higher rank group} will stand for a finite product of Higher rank simple groups. We warn the reader that for us, products of rank one groups such as $\mathrm{SL}_2(\R) \times \mathrm{SL}_2(\Q_p)$ are not of higher rank. We refer to \cite[Chapter I]{MR1090825} for the terminology. Note that real connected simple Lie group or real rank $\geq 2$ is more general than connected almost simple algebraic group of split rank $\geq 2$ over $\R$. It includes for example some groups with infinite center, as the infinite covering group of $\mathrm{Sp}_{2n}(\R)$.

Recall that a lattice in a locally compact group $G$ is a discrete subgroup $\Gamma$ such that $G/\Gamma$ carries a $G$-invariant Borel probability measure.

\begin{thm}\label{thm:strong_T_lattices} Every lattice in a higher rank group has strong property (T).
\end{thm}

Examples of lattices in higher rank groups include $\mathrm{SL}_{n}(\Z)$, $\mathrm{SL}_n(\F_p[t])$ and $\mathrm{SL}_n(\Z[\frac 1 p])$ for $n \geq 3$, or $\mathrm{Sp}_{2n}(\Z)$, $\widetilde{\mathrm{Sp}}_{2n}(\Z)$ (the preimage of $\mathrm{Sp}_{2n}(\Z)$ in the universal cover of $\mathrm{Sp}_{2n}(\R)$) and $\mathrm{Sp}_{2n}(\F_p[X])$ for $n \geq 2$. None of these examples is a cocompact lattice, so for all these cases Theorem \ref{thm:strong_T_lattices} is new.

When $\Gamma$ is a cocompact lattice in a locally compact group $G$, every representation of $\Gamma$ by bounded operators on a Hilbert (or Banach) space can be induced in a satisfactory way to a representation of $G$ by bounded operators on a Hilbert (Banach) space. This is what allows one to prove that (Banach) strong property (T) passes to cocompact lattices, see \cite{MR2423763}. As we shall explain in \S \ref{subsection:naive}, when $\Gamma$ is not cocompact, induction of representations which are not uniformly bounded does not behave well, and we do not see any reasonable way to define an induced Banach space representation. So the proof of Theorem \ref{thm:strong_T_lattices} \emph{does not proceed} by proving that strong property (T) passes to lattices. And we still have no idea whether such a statement is true (although amusingly, it is true that the negation of strong property (T) passes to lattices, see Corollary \ref{cor:NotStrTpasses_to_lattices}). This might appear at first sight a bit surprising, because it is now very well understood (this seems to go back at least to the proof of the normal subgroup theorem by Margulis) that, although they might not be cocompact, higher rank lattices are very much integrable (for example they are $L_p$-integrable for every $p<\infty$ in the sense of \cite{MR1767270}), and these good integrability properties enable to induce in a satisfactory way cocycles with values in isometric representations. The new idea that we introduce to overcome this difficulty is a form of induction of representation $\pi \colon \Gamma \to \mathrm{GL}(X)$ which, under some assumption on the integrability of the lattice and the growth rate of the norm of $\|\pi(\gamma)\|$, produces a representation-like object, where one is only allowed to compose once, and that we call a \emph{two-step representation}.
\begin{defn} A two-step representation of a topological group $G$ is a tuple $(X_0,X_1,X_2,\pi_0,\pi_1)$ where $X_0,X_1,X_2$ are Banach spaces and $\pi_i\colon G \to B(X_i,X_{i+1})$ are strongly continuous\footnote{\emph{i.e.} for every $x \in X_i$, the map $g \in G \mapsto \pi(g)x \in X_{i+1}$ is continuous, see Section \ref{sec:preliminaries}.} maps such that
  \[ \pi_1(gg') \pi_0(g'') = \pi_1(g) \pi_0(g' g'')\textrm{ for every }g,g',g'' \in G.\]

 In this case we will denote by $\pi \colon G \to B(X_0,X_2)$ the continuous map satisfying $\pi(gg')=\pi_1(g) \pi_0(g')$ for every $g,g' \in G$.
\end{defn}
It turns out that a form of strong property (T) also holds for two-step representations of higher rank groups. And this property passes to undistorted lattices (Theorem \ref{thm:exponentially_integrable_lattices}). This is the content of our main result Theorem \ref{thm:main}, which contains Theorem \ref{thm:strong_T_lattices} as a particular case.

Before stating it, we recall the notion of length function that we use, which contains as its main examples the word-length with respect to compact symmetric generating sets. A \emph{length function} on a locally compact topological group $G$ is a function $\ell \colon G \to \R^+$ such that
\begin{itemize}
\item $\ell$ is bounded on compact subsets of $G$.
\item $\ell(g^{-1})=\ell(g)$ for every $g \in G$.
\item $\ell(gh) \leq \ell(g)+\ell(h)$ for every $g,h \in G$.
\end{itemize}

The \emph{exponential growth rate} of a two-step representation $(X_0,X_1,X_2,\pi_0,\pi_1)$ with respect to a length function $\ell$ is
\[\max_{i=0,1} \limsup_{\ell(g) \to \infty} \frac{\log \|\pi_i(g)\|}{\ell(g)}.\]

We say that a pair $(G,\ell)$ of a locally compact group with a length function satisfies (*) if there exists $s,t,C >0$ and a sequence $m_n$ of \emph{positive} probability measures whose support is contained in $\{g|\ell(g) \leq n\}$ such that the following holds. Let $(X_0,X_1,X_2,\pi_0,\pi_1)$ be a two-step representation and $L$ a real number such that $X_1$ is a Hilbert space and $\|\pi_i(g)\|\leq L e^{s \ell(g)}$ for all $g \in G$ and $i \in \{0,1\}$. Then there is $P \in B(X_0,X_2)$ such that
\begin{equation}\label{eq:mn_Cauchy2}\|\pi(m_n) - P\| \leq  C L^2 e^{-tn},\end{equation}
and such that 
\begin{equation}\label{eq:lim_invariant2}\lim_n \|\pi(\delta_g \ast m_n \ast \delta_{g'}) - \pi(m_n)\|=0 \textrm{ for every }g,g'\in G.\end{equation}
Strong property (T) corresponds to the case when $X_0=X_1=X_2=X$ and $\pi$ is a representation. In that case \eqref{eq:lim_invariant2} is usually replaced by the equivalent property that $P$ is a projection on the space of invariant vectors $\{ x \in X | \pi(g) x = x \forall g \in G\}$, parallel to a $\pi(G)$-invariant complement subspace. The condition \eqref{eq:lim_invariant2} is nothing but a reformulation which remains meaningful in the above generality when there is no such thing as invariant vector or projection.

We say that $G$ satisfies (*) if $(G,\ell)$ satisfies (*) for every length function $\ell$, or equivalently if $G$ is compactly generated and $(G,\ell)$ satisfies (*) for the word-length function coming from a compact generating set. See Lemma \ref{lem:changing_length} for the equivalence.

\begin{thm}\label{thm:main} Every higher rank group or lattice in it satisfies the above property (*).
\end{thm}

Examples of maps $\pi$ as in (*) are when $\mathcal X$ is a topological vector space (for example the space of measurable functions on a manifold, or just a measure space) and $\pi \colon G \to \mathrm{GL}(\mathcal X)$ is a continuous representation of $G$ on $\mathcal X$ which a priori does not preserve any Banach space in $X$ (for example because of losing of derivatives, as in the Nash-Moser theorem, or of integrability). But there are three Banach spaces $X_0,X_1,X_2$ with $X_1$ a Hilbert space with continuous embeddings into $\mathcal X$ (for example encoding different scales of derivability or integrability) and such that $\pi(g)$ maps $X_i$ to $X_{i+1}$ with norm $\leq L e^{s \ell(g)}$. In that situation we can apply the conclusion of the theorem. In particular, we get, for every $x \in X_0$, that $\pi(m_n) x$ converges in the norm of $X_1$ (and hence in the topology of $\mathcal X$) to a $\pi(G)$-invariant vector. In this setting, property (*) has therefore to be seen as a procedure to systematically produce and locate invariant vectors in $\mathcal X$.

I would like to point out that, even if one is only interested in strong property (T) (so to representations on Hilbert spaces), it is crucial that in property (*) we allow arbitrary Banach spaces $X_0$ and $X_2$. Indeed, the induction procedure explained in Subsection \ref{subsection:induction}, which is the heart of this work, cannot produce Hilbert spaces but more general Banach spaces (namely Hilbert-space valued $L_p$ spaces for various values of $p$).

\subsection*{Banach space extensions}
Higher rank groups over non-archimedean local fields and their
cocompact lattices are known to satisfy strong Banach property (T)
with respect to every class of Banach spaces of nontrivial
(Rademacher) type \cite{MR2574023,liao} (see Section
\ref{sec:preliminaries} for the definitions). Moreover, this class is
essentially the optimal class. Although some partial results have been
obtained \cite{Salle2015,strongTsp4,dlMdlSAIF}, it is still not known
whether the same holds over the real numbers. I regard this question
as the main open problem on the subject, as a positive answer would
settle positively the conjecture in \cite{MR2316269} that every action
by isometries on a uniformly convex Banach space of a higher rank
lattice has a fixed point, and prove that the standard Cayley graphs
of $\mathrm{SL}_3(\Z/n\Z)$ form a family of superrexpanders.

In this article we also extend to all lattices the above mentionned results. 

To state the results, we introduce the following notion: if $\mathcal E$ is a class of Banach spaces we say that $G$ (respectively $(G,\ell)$) satisfies $(*_{\mathcal E})$ if in (*) the assumption that $X_1$ is a Hilbert space is replaced by $X_1 \in \mathcal E$.

The following result extends the results of Lafforgue and Liao \cite{MR2574023,liao}.


\begin{thm}\label{thm:main_Banach_valued_nonarch} Let $G$ be a higher rank simple group over a non-archimedean local field, or a lattice therein. Then $G$ satisfies $(*_{\mathcal E})$ for every class of Banach space $\mathcal E$ of nontrivial type.
\end{thm}
In particular, every lattice in a higher group over nonarchimedean local fields has strong property (T) with respect to every Banach space of nontrivial type.

In the real case the conditions we have to impose on the Banach spaces are a bit longer to state, but we believe that they are equivalent to having nontrivial type. For $n\geq 2$, denote by $\mathbb{S}^n$ the unit sphere in euclidean $\R^{n+1}$ and define a family $(T^{(n)}_\delta)_{\delta \in [-1,1]}$ of operators on $L_2(\mathbb{S}^n)$ by $T_\delta^{(n)} f(x)$ is the average of $f$ on $\{y \in \mathbb{S}^n  | \langle x,y \rangle = \delta\}$.

For $\theta \in \R/2\pi$, denote by $S_\theta$ the operator on $L_2(\mathrm{SU}(2))$ given by
\[ S_\theta f(u) = \int_{0}^{2\pi} f (\frac{1}{\sqrt{2}}\begin{pmatrix} e^{-i\theta} & -e^{i\varphi} \\ e^{-i\varphi} & e^{i\theta}\end{pmatrix} u) \frac{d \varphi}{2\pi}.\]

The following result extends the results of \cite{Salle2015,strongTsp4,dlMdlSAIF}. A version for general higher rank groups is stated as Theorem \ref{thm:main_Banach_general}.
\begin{thm}\label{thm:main_Banach_valued_real} Let $G$ be a connected simple Lie group with Lie algebra $\mathfrak g$ and $\Gamma \subset G$ be a lattice. Then both $G$ and $\Gamma$ have $(*_{\mathcal E})$ (and therefore strong (T) with respect to $\mathcal E$) if one of the following conditions holds:
  \begin{itemize}
  \item $\mathfrak{g}$ contains a Lie subalgebra isomorphic to $\mathfrak{sp}_4$, and there is $\alpha\in (0,1]$ and $C>0$ such that for every $X \in \mathcal E$
    \begin{equation}\label{eq:Stheta} \| S_\theta - S_{\frac{\pi}{4}}\|_{B(L_2(\mathrm{SU}(2);X))} \leq C |\theta - \frac \pi 4|^{\frac \alpha 4} \forall \theta \in [0,2\pi]\end{equation}
    and
    \begin{equation}\label{eq:Tdelta} \| T^{(2)}_\delta - T^{(2)}_0\|_{B(L_2(\mathbb{S}^n;X))} \leq C |\delta|^{\frac \alpha 2} \forall \delta \in [-1,1].\end{equation}
  \item $\mathfrak{g}$ contains a Lie subalgebra isomorphic to $\mathfrak{sl}_{3n-3}$ for $n \geq 2$, and there is $\alpha\in (0,1]$ and $C>0$ such that for every $X \in \mathcal E$,
\begin{equation}\label{eq:Tdeltan} \| T^{(n)}_\delta - T^{(n)}_0\|_{B(L_2(\mathbb{S}^n;X))} \leq C |\delta|^{\frac \alpha 2} \forall \delta \in [-1,1].\end{equation}
  \end{itemize}
\end{thm}
All the conditions \eqref{eq:Stheta}, \eqref{eq:Tdelta} and
\eqref{eq:Tdeltan} imply that $X$ has nontrivial type, and we believe that they are actually all equivalent. However, we only know that the condition when $\mathfrak{g}$ contains $\mathfrak{sp}_4$ is formally stronger that when it contains $\mathfrak{sl}_3$, and the condition \eqref{eq:Tdeltan} becomes formally weaker when $n$ grows. When $X$ is a Hilbert space, both \eqref{eq:Stheta} and \eqref{eq:Tdelta} hold with $\alpha=1$. Therefore, \eqref{eq:Stheta} and \eqref{eq:Tdelta} hold if $X$ is isomorphic to a subspace of an interpolation space $[X_0,X_1]_\alpha$ between a Hilbert space $X_1$ and an arbitrary Banach space $X$, or more generally if $X$ is $\theta$-Hilbertian (with $\theta=\alpha$) in the sense of \cite{MR2732331}. This holds in particular if $X$ is isomorphic a subspace of a superreflexive Banach lattice \cite{MR555306}. This includes for example all reflexive Sobolev spaces or Besov spaces.

Since every real simple Lie algebra of real rank $\geq 2$ contains a Lie subalgebra isomorphic to $\mathfrak{sl}_3$ or $\mathfrak{sp}_4$, the preceding implies that every higher rank lattice has strong (T) with respect to $\theta$-Hilbertian Banach spaces, but the results are more general as they include some non superreflexive spaces, for example those having good enough type and cotype exponents, see \cite{Salle2015}.

We end this introduction with another particular case of the above theorem (see \cite{dlMdlSAIF} for the proof that the assumption in Corollary \ref{cor:TBanachique} implies that \eqref{eq:Tdeltan} holds for $n$ large enough).
\begin{cor}\label{cor:TBanachique} Let $X$ be a Banach space for which there is $\beta < \frac 1 2$ and $C$ such that, for every integer $k$, every subspace of $X$ of dimension $k$ is at Banach-Mazur distance $\leq C k^{\beta}$ from $\ell_2^k$. There is $N_X$ such that every lattice in a connected simple Lie group of real rank $\geq N_X$ has strong property (T) with respect to $X$.
\end{cor}

Theorem \ref{thm:main}, as well as its Banach space generalizations, is proven in several steps. The first step is to prove the Theorem for the basic building blocks of higher rank groups, namely for $G=\mathrm{SL}_3(\F)$, $\mathrm{Sp}_4(\F)$ for $\F= \R,\Q_p$ or $\mathbf F_p(\!(t)\!)$, or $G=\widetilde{\mathrm{Sp}}_4(\R)$. This is achieved in Section \ref{sec:SL3Sp4}. The second step is to extend this to all higher rank groups in Section \ref{sec:all_higher_rank_groups}. The last step is to deal with lattices in such groups in Section \ref{sec:lattices}. A crucial ingredient is the fact that higher rank lattices are \emph{exponentially integrable}.
\subsection{Acknowledgements} I thank David Fisher and Tim de Laat for many interesting conversations and useful comments. I thank Fran\c{c}ois Maucourant for allowing me to include his proof of Theorem \ref{thm:measure_of_cusps}, much more elementary and general than my initial argument based on the reduction theory of $S$-arithmetic lattices. Even though this argument is no longer present in the final version of this work, I wish to thank Olivier Ta\"ibi and Kevin Wortman for their very patient explanations on the reduction theory of $S$-arithmetic lattices in positive characteristic.
\section{Preliminaries}\label{sec:preliminaries}
\subsection{Notation} 
If $G$ is a locally compact group, we will denote by $\Pc(G)$ the set
of all compactly supported Borel probability measures on $G$. To
lighten the notation, the convolution of probability measures $m_1,m_2
\in \Pc(G)$ will be written as $m_1m_2$. So
\[ \int f d(m_1 m_2) = \iint f(g_1 g_2) dm_1(g_1) dm_2(g_2).\]

We view $\Pc(G)$ as a set of linear forms on the space of continuous functions on $G$, and equip it with the restriction of the weak-* topology.

If $X,X'$ are Banach spaces, a map $\pi \colon G \to B(X,X')$ is called strongly continuous if for every $x \in X$, the map $g \in G \mapsto \pi(g) x \in X'$ is continuous. In that case, for every $m \in \Pc(G)$, we denote by $\pi(m) \in B(X,X')$ the operator $x\mapsto \int \pi(g) x dm(x)$ (Bochner integral). By applying the definitions, we readily obtain the following.
\begin{lemma}\label{lem:strong_continuity_of_pi_on_Pc} If $\pi \colon G \to B(X,X')$ is strongly continuous, then the map $\pi \colon \Pc(G) \to B(X,X')$ is still strongly continuous.
\end{lemma}

We recall the definition of Lafforgue's strong property (T).

Fix a left Haar measure $dg$ on $G$. If $\ell$ is a length function on locally compact group $G$, denote by $\mathcal C_{\ell}(G)$ the Banach algebra obtained by completion of convolution algebra $C_c(G)$ under the norm $\|f\|_\ell= \sup \{ \|\pi(f)\|\}$ where the supremum is over all strongly continuous representations $\pi$ of $G$ on a Hilbert space for which $\|\pi(g)\| \leq e^{\ell(g)}$ for every $g \in G$. As for measures, $\pi(f)$ is here the operator $x \mapsto \int f(g) \pi(g) x dg$.

For example, if $\ell=0$, we obtain $C^*(G)$, the full $C^*$-algebra
of $G$.

\begin{defn}\label{defn:strongT}(Lafforgue) A locally compact group $G$ has strong property (T) if for every length function $\ell$, there exists $s>0$ such that for every $c\geq 0$ the Banach algebra $\mathcal C_{s \ell +c}(G)$ has a \emph{Kazhdan projection}, \emph{i.e.} an idempotent $P$ such that $\pi(P)$ is a projection on the space of invariant vectors for every representation $\pi$ satisfying $\|\pi(g)\| \leq e^{s\ell(g)+c}$ for every $g \in G$
\end{defn}

A justification for this definition is the following well-known
characterization of property (T), which in particular asserts that the
particular case $\ell=0$, $c=0$ in Definition \ref{defn:strongT} is
equivalent to property (T).
\begin{prop}\label{prop:characterizations_T} For a locally compact group $G$, the following are equivalent.
  \begin{enumerate}
    \item $G$ has property (T).
    \item There is a compactly supported probability measure $\mu$ on
    $G$ such that, for every unitary representation $\pi$ of $G$ on a
    Hilbert, $\|\pi(\mu) - P_\pi\| \leq \frac 1 2$, where $P_\pi$ is the
    orthogonal projection on the space of invariant vectors and the
    norm is the operator norm on $G$.
  \item $G$ has a symmetric compact generating set $Q$ and there is a
    sequence $\mu_n$ of probability measures supported in $Q^n$ such
    that, for every unitary representation $\pi$ of $G$ on a Hilbert,
    $\|\pi(\mu_n) - P_\pi\| \leq 2^{-n}$.
  \item for every length function $\ell$ on $G$, there are constants
    $C,s>0$ and a sequence $\mu_n$ of probability measures supported
    in $\{g \in G |\ell(g) \leq n\}$ such that, for every unitary
    representation $\pi$ of $G$ on a Hilbert, $\|\pi(\mu_n) - P_\pi\|
    \leq C e^{-sn}$.
    \item $C^*(G)$ has a Kazhdan projection.
   \end{enumerate}
\end{prop}
\begin{rem}\label{rem:spectral_gap}
Actually this proposition holds \emph{representation-by-representation}:
given a unitary representation $\pi$ of a locally compact group $G$,
the following are equivalent:
\begin{itemize}
  \item $\pi$ has spectral gap in the sense that the orthogonal of the
    space of invariant vectors does not carry almost invariant
    vectors.
  \item there is a compactly supported probability measure $\mu$ on
    $G$ such that $\|\pi(\mu) - P_\pi\| \leq \frac 1 2$.
  \item there is a symmetric compact subset $Q \subset G$, and a
    sequence of probability measures $\mu$ on $Q^n$ such that
    $\|\pi(\mu_n) - P_\pi\| \leq 2^{-n}$.
  \item for every length function $\ell$
    on $G$, there are constants $C,s>0$ and a sequence $\mu_n$ of
    probability measures supported in $\{g \in G |\ell(g) \leq n\}$
    such that $\|\pi(\mu_n) - P_\pi\| \leq C e^{-sn}$.
\end{itemize}
If one defines correctly a Kazhdan projection for arbitrary
Banach-algebra completions of $C_c(G)$ (see \cite{{dlSlocalKazhdan}}),
these definitions are in turn equivalent to the existence of a Kazhdan
projection for the completion of $C_c(G)$ for the norm $\|f\| =
\|\pi(f)\|$.
\end{rem}

If $\mathcal E$ is a class of Banach spaces, one can denote similarly
by $\mathcal C_{\ell,\mathcal E}(G)$ the Banach algebra obtained by
completion of $C_c(G)$ under the norm $\|f\|_{\ell,\mathcal E}= \sup
\{ \|\pi(f)\|\}$ where the supremum is over all strongly continuous
representations $\pi$ of $G$ on a Banach space in $\mathcal E$ for
which $\|\pi(g)\| \leq e^{\ell(g)}$ for every $g \in G$, and define
Banach strong property (T) with respect to $\mathcal E$ as strong
property (T) by replacing $\mathcal C_{s \ell +c}(G)$ by $\mathcal
C_{s \ell +c,\mathcal E}(G)$.

Recall that a Banach space $X$ has nontrivial Rademacher type (or simply nontrivial type) if there exists $p>1$ and a real number $T$ such that
\begin{equation}\label{eq:typep} \left(\mathbb E \| \sum_i \varepsilon_i x_i\|^p\right)^{\frac 1 p} \leq T \left(\sum_i \|x_i\|^p\right)^{\frac 1 p}\end{equation}
for every finite sequence $x_i$ in $X$, where $\varepsilon_i$ are iid random variables uniformly distributed in $\{-1,1\}$. This is equivalent to the fact that $\ell^1$ is not finitely representable in $X$: there is $N>0$ and $c>1$ such that every linear map $u$ between $\ell^1_N$ and every $N$-dimensional subspace of $X$ satisfies $\|u\| \|u^{-1}\| \geq c$.

More generally a class of Banach spaces $\mathcal E$ has nontrivial type if there exists $p>1$ and $T<\infty$ such that \eqref{eq:typep} holds for every $X \in \mathcal E$ and every finite sequence $(x_i)$ in $X$, or equivalently if $\ell^1$ is not finitely representable in $\mathcal E$.

\subsection{Why the naive attempt does not work}\label{subsection:naive}
We now explain why the classical notion of induction of representations, that we first recall, is not well-suited to induce Strong (T) to non-cocompact lattices.

Let $\Gamma$ be a lattice in a locally compact group $G$. Let $\Omega$ be a Borel fundamental domain for $G/\Gamma$: $\Omega$ is a subset of $\Omega$, belonging to the Borel $\sigma$-algebra, and such that $(\omega,\gamma) \in \Omega \times \Gamma \mapsto \omega \gamma \in G$ is a bijection.

Let $\pi$ be a representation of $\Gamma$ on a Hilbert or Banach space $X$. Consider the topological vector space $\widetilde X$ of (Bochner-measurable) functions $f\colon G \to X$ satisfying $f(g \gamma) = \pi(\gamma)^{-1} f(g)$, moded out by functions that vanish outside of a negligeable set. Make $G$ act on this space by left translation: $\widetilde \pi(g) f(h) = f(g^{-1}h)$.

It is natural to consider the Hilbert space of such functions satisfying moreover
\[ (\int_\Omega \|f(\omega)\|_X^2 d\omega)^{\frac 1 2}<\infty.\]
This space is naturally identified with $L_2(\Omega;X)$. Under this identification, if $g \omega=(g \cdot \omega) \alpha(g,\omega)$ is the unique decomposition of $g\omega$ in $G=\Omega \Gamma$, then $\widetilde \pi(g)$ reads as
\[ (\widetilde \pi(g) f)(\omega) = \pi(\alpha(g^{-1},\omega)^{-1}) f(g^{-1} \cdot \omega) \ \forall g \in G, \omega \in \Omega.\]
The problem that occurs is that $\widetilde \pi(g)$ preserves $L_2(\Omega;X)$ if and only if the function $\omega \mapsto \|\pi(\alpha(g^{-1},\omega)^{-1}\|$ is essentially bounded on $\Omega$:
\begin{lemma} The norm of $\widetilde \pi(g)$ on $L_2(\Omega;X)$ is equal to the essential supremum of $\omega \mapsto \|\pi(\alpha(g^{-1},\omega)^{-1})\|$.
\end{lemma}
\begin{proof}
Let $C_g$ be the essential supremum of $\|\pi(\alpha(g^{-1},\omega)^{-1})\|$.
  The inequality $\|(\widetilde \pi(g) f)(\omega)\| \leq \|\pi(\alpha(g^{-1},\omega)^{-1})\| \|f(g^{-1} \cdot \omega)\|$ implies that 
  \[ \|(\widetilde \pi(g) f)(x)\|_{L_2(\Omega;X)}^2 \leq C_g^2 \int_\Omega \|f(g^{-1} \cdot \omega)\|^2 d\omega =C_g^2 \|f\|_{L_2(\Omega;X)}^2\]
  because $\omega \mapsto g^{-1} \cdot \omega$ preserves the measure on $\Omega$.

  For the other direction, for $\gamma \in \Gamma$, denote $A=\{\omega \in \Omega| \alpha(g^{-1},\omega)^{-1}=\gamma\} = \Omega \cap g \Omega \gamma^{-1}$. If $A$ has positive measure, then for every $x \in X$ we can consider $f = \chi_A x$. It has norm $|A|^{\frac 1 2} \|x\|$, and its image $\chi_{g A} \pi(\gamma) x$ has norm $|A|^{\frac 1 2} \|\pi(\gamma) x\|$. Taking the supremum over $x$ yields the inequality $\|\widetilde \pi(g)\| \geq \|\pi(\gamma)\|$. Taking the supremum over all $g$ such that $\Omega \cap g \Omega \gamma^{-1}$ has positive measure prove that $\|\widetilde \pi(g)\|$ is larger than or equal to $C_g$.
\end{proof}
So in general $\widetilde \pi$ is not a representation by bounded operators unless $\Gamma$ is cocompact or $\pi$ is a uniformly bounded representation. There does not seem to be any other reasonable pseudo-norm on $\widetilde X$ for which $\widetilde \pi(g)$ is by bounded operators. There is always the pseudo norm $\|f\|=\infty$ $\forall f \neq 0$, but this is clearly unreasonable. We do not give a precise meaning to ``reasonable'', but it should at least remember the whole representation, for example by giving finite norm, for every $x \in X$, to the constant function equal to $x$ on $\Omega$. 

We mention however the construction in \cite{MR3621428} where a pseudo-norm is constructed on $\widetilde X$, which, under the assumption that the bounded cohomology $H^1_b(\Gamma;\pi)$ is non zero, gives rise to a nonzero space for which $H^1_b(G;\widetilde X)$ is also non zero. 

\subsection{Comparing Theorem \ref{thm:strong_T_lattices} and \ref{thm:main}}
We recall that $\mathcal C_{\ell,\mathcal E}(G)$ has a Kazhdan projection if and only if there is a sequence $m_n$ of signed\footnote{It is not known in general if $m_n$ can be taken to be positive.} compactly supported measures on $G$ with $\int 1 dm_n=1$ and $C>0$ such that $\|m_n - m_{n+1}\|_{\ell,\mathcal E} \leq C e^{-n}$ and such that $\lim_n\| g m_n -m_n\|_{\ell,\mathcal E} = 0$ for every $g \in G$. Moreover, $m_n$ can be taken to be of the form $(m_1)^{n}$ (the $n$-th convolution power of $m_1$). In particular $m_n$ is supported in $\{g|\ell(g) \leq nR\}$ if $m_1$ is supported in $\{g |\ell(g) \leq R\}$. Also, if $\mathcal E$ is stable by duality and subspaces, then the preceding implies that $\lim_n\| m_ng -m_n\|_{\ell,\mathcal E} = 0$ for every $g \in G$. For details, we refer to \cite{dlSlocalKazhdan} where these assertions were established. 

Hence in the particular case when $X_0=X_1=X_2=X$ and $\pi$ is a representation on $X$, property (*) for $(G,\ell)$ says a bit more than that $\mathcal C_{s\ell +c}(G)$ has a Kazhdan projection for every $c>0$~: first it says that $m_n$ can be taken independant from $c$, that $C = O(e^{2c})$ and most importantly that $m_n$ can be taken to be positive.

\subsection{Basic properties}
The first basic lemma implies that to prove Theorem \ref{thm:main}, it is enough to consider the word-length function with respect to some compact symmetric generating set (which exists because (*) is formally stronger than property (T), which already implies compact generation \cite{MR2415834}), or any other length function quasi-isometric to it. Indeed, if $\ell$ is any length function on a locally compact compactly generated group $G$, and $Q$ is a compact symmetric generating set for $G$ with associated length function $\ell_Q$, then there is $a>0$ such that $\ell \leq a\ell_Q$. Namely the supremum of $\ell$ on $Q$.
\begin{lemma}\label{lem:changing_length} Let $\ell,\ell'$ be two length functions, and $a,b >0$ such that $\ell' \leq a\ell +b$. If $(G,\ell)$ has $(*_{\mathcal E})$ then so does $(G,\ell')$.
\end{lemma}
\begin{proof} If $(G,\ell)$ has $(*_{\mathcal E})$ with $s,t,C$ and $m_n$, it is immediate that $(G,\ell')$ has $(*_{\mathcal E})$ with $\frac{s}{a},t/a,C'$ and $m_{\lfloor (n-b)/a\rfloor}$, with $C'=C e^{(2sb+ta+tb)/a}$.
\end{proof}

In each section of the paper, the proof of (*) or $(*_{\mathcal E})$ is divided in two parts: one first finds a sequence $m_n$ such that, if $s>0$ is small enough and $\pi$ is as in (*), then $\pi(m_n)$ converges as in \eqref{eq:mn_Cauchy2}. Then one proves that \eqref{eq:lim_invariant2} also holds. This second part is always much harder than the first. The next remark shows that it is not necessary to prove the norm convergence in \eqref{eq:lim_invariant2}.

\begin{rem}\label{rem:improvement_invariance} In (*), condition \eqref{eq:lim_invariant2} can be strengthened (or weakened). Indeed, once one knows that \eqref{eq:mn_Cauchy2} holds for every $\pi$ as in (*), then for any $\mu_1,\mu_2 \in \Pc(G)$ one can apply it to the new $\pi'$ given by $\pi'(m) = \pi(\mu_1 m \mu_2)$. Indeed, this $\pi'$ satisfies the same assumptions, but with $L$ replaced by $L e^{\frac{s}{2}(R_1+R_2)}$ if the support of $\mu_i$ is contained in $\{g| \ell(g) \leq R_i\}$. And so there is  ${}_{\mu_1}P_{\mu_2} \in B(X_0,X_2)$ such that for every $n$,
\begin{equation} \|\pi(\mu_1 m_n \mu_2) - {}_{\mu_1}P_{\mu_2} \| \leq CL^2 e^{s(R_1+R_2)-tn}.\end{equation}
And so \eqref{eq:lim_invariant2} is equivalent to each of the following properties:
\begin{itemize} \item ${}_{\delta_{g}}P_{\delta_{g'}} = P$.
\item $\|\pi(gm_ng') - \pi(m_n)\| \leq 2CL^2e^{s(\ell(g)+\ell(g'))-tn}$.
\item for every $x \in X_0$, $\lim_n \|\pi(gm_ng')x - \pi(m_n)x \|=0$.
\item for every $x \in X_0$, $\lim_n \pi(gm_ng')x - \pi(m_n)x =0$ weakly.
\end{itemize}
\end{rem}

\begin{lemma}\label{lem:direct_products} If $(G_1,\ell_1)$ and $(G_2,\ell_2)$ have (*) (respectively $(*_{\mathcal E})$) then so does $(G_1\times G_2,\ell)$ where $\ell(g_1,g_2) = \max(\ell_1(g_1),\ell_2(g_2))$.
\end{lemma}
\begin{proof} For $i=1,2$, let $s_i,t_i,C_i,m_n^{(i)}$ be as in $(*_{\mathcal E})$ for $G_i$. Define $m_n = m_n^{(1)} \otimes m_n^{(2)}$. By definition it is a probability measure supported in $\{g \in G_1 \times G_2 | \ell(g)  \leq n\}$.

  Let $\pi \colon G_1 \times G_2 \to B(X_0,X_2)$ be as in $(*_{\mathcal E})$ for $C,s$. We claim that the conclusion of $(*_{\mathcal E})$ holds if $s>0$ is small enough.

  We can compute
  \begin{multline*} \| \pi(m_n) - \pi(m_{n+1})\| \leq \| \pi(m_n^{(1)} \otimes m_n^{(2)} ) - \pi( m_n^{(1)} \otimes m_{n+1}^{(2)}) \| \\+ \| \pi(m_n^{(1)} \otimes m_{n+1}^{(2)} ) - \pi( m_{n+1}^{(1)} \otimes m_{n+1}^{(2)}) \|.\end{multline*}
  By \eqref{eq:mn_Cauchy2} applied to the map $g_2 \in G_2 \mapsto \pi(m_n^{(1)} \otimes \delta_{g_2})$, if $s \leq s_2$ the first term is dominated by $2 C_2 L^2 e^{2sn-t_2 n}$. Similarly, if $s \leq s_1$ the second term is dominated by $2 C_1 L^2 e^{2s(n+1)-t_1 n}$. So if $s = \min(\frac{t_1}{3},\frac{t_2}{3},s_1,s_2)$ then 
  \[ \| \pi(m_n) - \pi(m_{n+1})\| \leq (2C_1e^{2s}+2C_2) L^2 e^{-sn}.\]
  This implies that $\pi(m_n)$ is Cauchy and that  \eqref{eq:mn_Cauchy2} holds with $t=s$ and $C= \frac{2C_1e^{2s}+2C_2}{1-e^{-s}}$. The validity of \eqref{eq:lim_invariant2} follows with a similar proof, taking into account Remark \ref{rem:improvement_invariance}.
\end{proof}

\section{Proof of Theorem \ref{thm:main} for $\mathrm{SL}_3$ and $\mathrm{Sp}_4$}\label{sec:SL3Sp4}

The aim of this section is to prove Theorem \ref{thm:main}, Theorem \ref{thm:main_Banach_valued_nonarch} and Theorem \ref{thm:main_Banach_valued_real} for $\mathrm{SL}_3$, $\mathrm{Sp}_4$ and $\widetilde{\mathrm{Sp}}_4(\R)$, and Theorem \ref{thm:main_Banach_valued_real} for $\mathrm{SL}_{3n-3}$. As we shall see, the proofs use the same two main ingredients as the proofs of strong property (T)~: one is harmonic analysis in the maximal compact subgroups, and the other is a careful exploration process of the Weyl chambers using some elementary moves coming from the maximal compact subgroup. These ingredients are the same, but they are combined in a different way. We will give a complete and essentially self-contained proof for $\mathrm{SL}_3$ and be much more sketchy for the other groups. This allows us to divide the length of the paper by a factor of at least $2$, and we believe that the interested reader will be able to fill the details. The proof for $\mathrm{SL}_3(\F)$ is essentially independant from the local field, but for a better readability we have chosen to first focus on the real case, and then explain the small changes that one has to make to deal with non-archimedean local fields.

\subsection{Case of $\mathrm{SL}_3(\R)$}\label{subsection:SL3R}
We prove the theorem for $G=\mathrm{SL}_3(\R)$. We denote by $K=\mathrm{SO}(3) \subset G$ the maximal compact subgroup. By Lemma \ref{lem:changing_length} it is enough to prove the theorem for the length function $\ell(g) = \max(\log \|g\|,\log \|g^{-1}\|)$, where $\|\cdot\|$ is the norm induced from the natural $K$-invariant euclidean norm on $\R^3$:
\[\|(s_1,s_2,s_3)\| = \left(s_1^2+s_2^2+s_3^2\right)^{\frac 1 2}.\]
More precisely, we will prove that $(\mathrm{SL}_3(\R),\ell)$ has (*) with the parameters $s<\frac 1 4, t=\frac 1 2 -2s,C=\frac{100}{1-4s}$ and $m_n$ any $K$-biinvariant probability measure on $\{g|n-1 \leq \ell(g) \leq n\}$.

Let $\pi$ as in Theorem \ref{thm:main} with $s< \frac 1 4$. Denote by $d$ the distance on the compactly supported Borel probability measures on $G$ defined by \[d(m,m') = \|\pi(m)-\pi(m')\|_{B(X_0,X_2)}.\] The following lemma lists the properties of $d$. In (\ref{item:local_estimate}) and in the rest of the proof, $\lambda$ stands for the left regular representation of $K$. It is the representation on $L_2(K)$ given by $\lambda(k) f(\cdot) = f(k^{-1} \cdot)$ for every $k \in K$ and $f \in L_2(K)$. The crucial property is (\ref{item:local_estimate}). It is an incarnation for the compact group $K$ of more general phenomenon: uniformly bounded $2$-step representations of \emph{amenable groups} are governed by the left regular representation. We do not elaborate on this as all we need is (\ref{item:local_estimate}).
\begin{lemma}\label{lem:properties_of_d} The distance $d$ has the following properties.
  \begin{enumerate}
  \item\label{item:convexity} (Convexity) For every $m_1,m_2 \in \Pc(G)$,
    \[d(\frac{m_1 +m_2}{2},\frac{m'_1 +m'_2}{2}) \leq \frac 1 2(d(m_1,m'_1)+d(m_2,m'_2)).\]
  \item\label{item:lsc} (Lower-semicontinuity) If $Q \subset G$ is compact and $m_i$ (resp. $m'_i$) is a net of probability measures supported in $Q$ and converging weak-* to $m$ (resp. $m'$), then $d(m,m') \leq \liminf_i d(m_i,m'_i)$.
  \item\label{item:local_estimate} If $\mu,\mu'$ are probability measures on $K$ and $g_1,g_2 \in G$ then \[d(\delta_{g_1} \mu \delta_{g_2},\delta_{g_1}  \mu'  \delta_{g_2}) \leq L^2 e^{s\ell(g_1)+s\ell(g_2)} \|\lambda(\mu-\mu')\|_{B(L_2(K))}.\]
  \end{enumerate}
  
  \end{lemma}
\begin{proof} Property (\ref{item:convexity}) is obvious, and (\ref{item:lsc}) is immediate from the strong continuity of $\pi_0,\pi_1$ (and hence of $\pi$), see Lemma \ref{lem:strong_continuity_of_pi_on_Pc}. For (\ref{item:local_estimate}), consider $x \in X_0$ and $y \in X_2^*$. For every $k \in K$ define $F(k) =  \pi_0(k^{-1} g_2) x \in X_1$ and $H(k) =  \pi_1(g_1 k)^* y \in X_1^*$. For $k_1,k_2 \in K$ we have
\[ \langle H(k_1),F(k_2)\rangle = \langle y, \pi(g_1 k_1 k_2^{-1} g_2) x\rangle.\]
We view the continuous function $F$ as an element of $L_2(K;X_1)$. Its norm is less than $\sup_{k \in K} \|\pi_0(k^{-1} g_2) x\| \leq L e^{s\ell(g_2)} \|x\|$. Similarly, we view $H$ in the topological dual $L_2(\Omega;X_1)^*$, and it has norm $\leq L e^{s\ell(g_1)} \|y\|_{X_2^*}$. We can compute
  \[ \langle H, (\lambda(\mu) \otimes \mathrm{id}_{X_1})(F)\rangle = \iint_K \langle H(k_1),F(k_2^{-1} k_1) \rangle d\mu(k_2) dk_1 = \langle y, \pi( \delta_{g_1}  \mu  \delta_{g_2}) x\rangle.\]
  One deduces
  \[|\langle y, \pi( \delta_{g_1}  (\mu-\mu')  \delta_{g_2}) x\rangle| \leq \| \lambda(\mu-\mu') \otimes \mathrm{id}_{X_1} \| \|F\|_{L_2(K;X_1)} \|H\|_{L_2(K;X_1)^*},\]
  which is less than \[\|\lambda(\mu-\mu')\|_{B(L_2(K))} L^2 e^{s\ell(g_1)} e^{s\ell(g_2)} \|x\| \|y\|\]
  because $X_1$ is a Hilbert space. The lemma follows by taking the supremum over all $x$ and $y$ in the unit balls of $X_1$ and $X_2^*$ respectively.
  \end{proof}
\begin{rem}\label{rem:properties_of_d_Banach} If we are in the setting of property $(*_{\mathcal E})$ (that is if $X_1$ is a Banach space in $\mathcal E$), then Lemma \ref{lem:properties_of_d} and its proof still holds, with (\ref{item:local_estimate}) replaced by \[d(\delta_{g_1} \mu \delta_{g_2},\delta_{g_1}  \mu'  \delta_{g_2}) \leq L^2 e^{s\ell(g_1)+s\ell(g_2)} \|\lambda(\mu-\mu')\|_{B(L_2(K;X_1))}.\]
  \end{rem}
We shall prove Theorem \ref{thm:main} for $\mathrm{SL}_3(\R)$ in the generality given by the previous lemma. So let $d$ be a distance on the compactly supported probability measures on $G$ satisfying the three conditions (\ref{item:convexity}), (\ref{item:lsc}) and (\ref{item:local_estimate}) in the previous lemma. 

We say that a probability measure $\nu$ on a compact group $K$ is \emph{admissible} if it is absolutely continuous with respect to the Haar measure on $K$ and if the Radon-Nikodym derivative is strictly positive and is a coefficient of a finite dimensional representation of $K$. We say that $\nu$ is central if it belongs to the center of the convolution algebra of Borel measures on $K$.

\begin{prop}\label{prop:local_SL3} Denote by $\lambda_K$ the Haar probability measure on $K$, seen as a probability measure on $G$. There exists $C>0$ such that if $s<\frac 1 4$ and $t:= (\frac 1 2 - 2s)>0$, then
  \begin{equation}\label{eq:mg_Cauchy} d(\lambda_K \delta_g  \lambda_K, \lambda_K \delta_{g'}  \lambda_K) \leq \frac{C}{1-4s} L^2 \max(e^{-t\ell(g)},e^{-t\ell(g')}).
  \end{equation}
  For every admissible and central probability measure $\nu$ on $K$, there is $C(\nu) \in \R$ such that for every $g \in G$,
  \begin{equation}\label{eq:nugHaar_close_to_mg} d(\nu \delta_g  \lambda_K, \lambda_K \delta_{g}  \lambda_K) \leq C(\nu) L^2 e^{-t\ell(g)}. \end{equation}
\end{prop}
This proposition easily implies the Theorem. Indeed, the first half implies that there is $P$ in the completion of $(\Pc(G),d)$ (which is contained in $B(X_0,X_2)$ in our case) such that $d(\lambda_K \delta_g \lambda_K,P) \leq \frac{C}{1-4s}L^2 e^{-t\ell(g)}$.

More generally if $m_0 \in \Pc(G)$, applying the same to $d'(m,m') = d(m_0  m,m_0  m')$ (which satisfies the same assumptions than $d$ with $L^2$ replaced by $L^2 e^{sR}$ if $m_0$ is supported in $\{g,\ell(g)=R\}$), we obtain ${}_{m_0}P$ in the completion of $(\Pc(G),d)$ such that
\begin{equation}\label{eq:Cauchy_for_d_translated}
  d(m_0\lambda_K \delta_g \lambda_K,{}_{m_0}P) \leq \frac{C}{1-4s} L^2 e^{sR-t\ell(g)}.
  \end{equation}
\begin{lemma}\label{lemma:lsc} The map $m_0 \mapsto {}_{m_0}P$ is lower-semicontinuous.
\end{lemma}
\begin{proof} Let $m_i$ (resp. $m'_i$), $i \in I$ be a net converging weak-* to $m_0$ (resp. $m'_0$) and supported in a common compact subset of $G$, say $\{g,\ell(g) \leq R\}$. For every $g \in G$, \eqref{eq:Cauchy_for_d_translated} yields 
  \[ d({}_{m_0}P,{}_{m'_0}P) \leq 2\frac{C}{1-4s} L^2 e^{sR-t\ell(g)} + d(m_0\lambda_K \delta_g \lambda_K,m'_0\lambda_K \delta_g \lambda_K).\]
  By the lower-semicontinuity of $d$, we deduce
  \[ d({}_{m_0}P,{}_{m'_0}P) \leq 2\frac{C}{1-4s} L^2 e^{sR-t\ell(g)} + \liminf_i d(m_i\lambda_K \delta_g \lambda_K,m'_i\lambda_K \delta_g \lambda_K),\]
  which (by \eqref{eq:Cauchy_for_d_translated}) is bounded above by
  \[ 4\frac{C}{1-4s} L^2 e^{sR-t\ell(g)} + \liminf_i d({}_{m_i}P,{}_{m'_i}P).\]
  The lemma follows by making $\ell(g) \to \infty$.
\end{proof}


The second half of the proposition implies that $d(\nu \delta_g  \lambda_K,P) \leq (\frac{C}{1-4s}+C(\nu))L^2 e^{-t\ell(g)}$ if $\nu$ is an admissible and central probability measure on $K$. Using the convexity (\ref{item:convexity}) and the lower-semicontinuity  (\ref{item:lsc}) of $d$ we get that for $g_1 \in G$,
\[ d(\nu  \delta_{g_1} \lambda_K  \delta_g  \lambda_K,P) \leq (\frac{C}{1-4s}+C(\nu))L^2 e^{-t\ell(g)+t\ell(g_1)}.\]
Making $\ell(g) \to \infty$, we obtain ${}_{\nu\delta_{g_1}}P=P$. By the Peter-Weyl theorem we can find a sequence $\nu_n$ of admissible and central probability measures on $K$ converging weak-* to $\delta_1$. By Lemma \ref{lemma:lsc} we deduce that
\[ d(P,{}_{\delta_{g_1}}P) \leq \liminf_i d({}_{\nu_i}P,{}_{\nu_i\delta_{g_1}}P) = 0.\]
To summarize, if $m_g$ is the $K$-biinvariant probability measure on $KgK$, we have proven that $d(m_g,P) \leq \frac{C}{1-4s}L^2 e^{-t\ell(g)}$ and $\lim_g d(\delta_{g_1} m_g ,P)=0$ for every $g_1 \in G$. If we consider the distance $(m,m') \mapsto d(\check m,\check m')$ for $\check m$ the image of $m$ by the inverse map\footnote{This new distance satisfies the same hypotheses as $d$.}, we also have $\lim_g d(m_g \delta_{g_2}  ,P)=0$ for every $g_2 \in G$, and hence $\lim_g d(\delta_{g_1} m_g \delta_{g_2}  ,P)=0$. This proves the theorem.

It remains to prove Proposition \ref{prop:local_SL3}. As in Lafforgue's original proof \cite{MR2423763} (see also the exposition in \cite{Salle2015}), the proof is based on the harmonic analysis in the compact group $K$.

We introduce the subgroups $U,\widetilde U \subset K$ of block-diagonal matrices
\[ U = \left\{ \begin{pmatrix} * & 0 &0\\0&*&*\\0&*&*\end{pmatrix} \right\} \cap K.\]
\[ \widetilde U = \left\{ \begin{pmatrix} *&*&0\\ *&*&0\\0&0&*\end{pmatrix} \right\} \cap K.\]
$U$ and $\widetilde U$ are both isomorphic to $\mathrm{O}(2)$.

For $\delta \in [0,1]$ we introduce the following matrix $k_\delta \in K$ with entry $(1,1)$ equal to $\delta$~:
\[ k_\delta = \begin{pmatrix} \delta & -\sqrt{1-\delta^2} & 0 \\ \sqrt{1-\delta^2} & \delta& 0\\0&0&1 \end{pmatrix}.\]

The fundamental inequality proven by Lafforgue in \cite[Lemme 2.2]{MR2423763} is that 
\begin{equation}\label{eq:T_delta_Hoelder}\|\iint_{U \times U} \lambda(u k_\delta u')-\lambda(u k_0 u') du du'\|_{B(L_2(K))} \leq 2 |\delta|^{\frac 1 2}.\end{equation}

This implies more generally that if $\mu_1,\mu_2$ are admissible probability measures on $U$, then
\begin{equation}\label{eq:twisted_Tdelta_Hoelder} \|\iint_{U \times U} \lambda(u k_\delta u')-\lambda(u k_0 u') d\mu_1(u) d\mu_2(u')\|_{B(L_2(K))} \leq C(\mu_1,\mu_2) |\delta|^{\frac 1 2}.\end{equation}
See \cite{MR2423763} or \cite[Proposition 2.1]{dlMdlSAIF}.

For $\alpha,\beta,\gamma \in \R$ with $\alpha+\beta+\gamma=0$, we denote
\[ D(\alpha,\beta,\gamma) = \begin{pmatrix} e^\alpha & 0&0\\ 0&e^\beta &0\\ 0&0&e^\gamma\end{pmatrix}.\]
For $\alpha \geq 0$, we simply write $D_\alpha$ for $D(2\alpha,-\alpha,-\alpha)$. It has norm $e^{2\alpha}$ and $\ell(D_\alpha) = 2\alpha$.

We start with the proof of \eqref{eq:mg_Cauchy}. Denote by $\Lambda$ the Weyl chamber, that is $\Lambda = \{(a_1,a_2,a_3) \in \R^3 | a_1\geq a_2 \geq a_3, a_1+a_2+a_3=0\}$. For $(a_1,a_2,a_3) \in \Lambda$ denote
\[ c(a_1,a_2,a_3) = \lambda_K \delta_{D(a_1,a_2,a_3)} \lambda_K.\]
By the $KAK$-decomposition, \eqref{eq:mg_Cauchy} is equivalent to the inequality \[d(c(a_1,a_2,a_3),c(a_1',a_2',a_3')) \leq \frac{C}{1-4s} L^2 \max(e^{-t\max(a_1,-a_3)},e^{-t\max(a_1',-a_3')}).\]
Since $D_\alpha$ commutes with every element of $U$, we can write
\[ \lambda_K \delta_{D_\alpha k_\delta D_\alpha} \lambda_K = \lambda_K \delta_{D_\alpha}\lambda_U \delta_{k_\delta}\lambda_U \delta_{D_\alpha} \lambda_K.\]
It therefore follows from \eqref{eq:T_delta_Hoelder} and the properties of $d$ in Lemma \ref{lem:properties_of_d} that
\[ d( \lambda_K \delta_{D_\alpha k_\delta D_\alpha} \lambda_K, \lambda_K \delta_{D_\alpha k_0 D_\alpha} \lambda_K) \leq 2 L^2 e^{4s\alpha} |\delta|^{\frac 1 2}.\]
To make this formula more readable we compute the KAK decomposition of $D_\alpha k_\delta D_\alpha$. For $\delta=0$, we have
\begin{equation}\label{eq:KAK_dec_for_k_0}
  D_\alpha k_0 D_\alpha = D(\alpha,\alpha,-2\alpha) k_0.
\end{equation}
For $\delta\neq 0$ we have the lemma.
\begin{lemma}\label{lem:spheres_distorted} For every $r \in [\alpha,4\alpha]$ there are $\delta \in [0,1]$ and $u_{r,\alpha},u'_{r,\alpha} \in \widetilde U$ such that $\delta\leq e^{r-4\alpha} \leq 1$ and
\[D_\alpha k_\delta D_\alpha = u_{r,\alpha} D(r,2\alpha-r,-2\alpha)u'_{r,\alpha}.\]
\end{lemma}
\begin{proof}
  For $\delta\neq 0$, $g=D_\alpha k_\delta D_\alpha$ is block diagonal with one eigenvalue $e^{-2\alpha}$ and another block of the form $D k D$ for $D = \mathrm{diag}(e^{2\alpha},e^{-\alpha})$ and $k$ an isometry. In particular $\|g^{-1}\| = e^{2\alpha}$. If we define $r_\alpha (\delta) \in [0,\infty)$ by $\|g\| = e^{r_\alpha(\delta)}$ we therefore have that $g\in \widetilde U D(r_\alpha(\delta),2\alpha - r_\alpha(\delta),-2\alpha)\widetilde U$. By saying that the norm of $g$ is larger that the absolute value of its $(1,1)$ entry we get the desired inequality $\delta e^{4\alpha}\leq e^{r_\alpha(\delta)}$. It remains to show that $r_\alpha$ is surjective. But $r_\alpha$ is continuous on the interval $[0,1]$ so its image contains the interval $[r_\alpha(0),r_\alpha(1)] = [\alpha,4\alpha]$. We do not need it, but it is not hard to check that $r_\alpha$ is actually bijective from $[0,1]$ onto $[\alpha,4\alpha]$.
\end{proof}

In particular for every $(a_1,a_2,a_3),(a'_1,a'_2,a'_3) \in \Lambda$ satisfying $a_3=a'_3$, by applying the preceding lemma with $-2\alpha = a_3=a'_3$ we have that,
\begin{equation}\label{eq:Hoelder_continuity_c}
 d( c(a_1,a_2,a_3),c(a'_1,a'_2,a'_3))| \leq 4L^2 (e^{\frac{a_1}{2}+(1-2s)a_3}+e^{\frac{a'_1}{2}+(1-2s)a_3'}).
\end{equation}

Notice that if $a_2 \geq -1$ we have $\frac{a_1}{2}+(1-2s)a_3 = \frac{a_1+a_2+a_3}{2}+(\frac 1 2 -2s)a_3-\frac{a_2}{2} \leq \frac 1 2 + (\frac 1 2 - 2s)a_3$. Therefore \eqref{eq:Hoelder_continuity_c} implies
\begin{equation}\label{eq:horizontal_estimates}
  d( c(a_1,a_2,a_3),c(a'_1,a'_2,a'_3)) \leq 14 L^2  e^{(\frac{1}{2}-2s)a_3}\textrm{ if $a_3=a'_3$ and }a_2,a'_2 \geq -1.\end{equation}
If we apply the same for the distance $d'(m,m') = d(\rho_* m,\rho_*m')$ for $\rho$ the Cartan automorphism
\[ g \mapsto \begin{pmatrix} 0&0&1\\0&1&0\\1&0&0\end{pmatrix}(g^{-1})^t \begin{pmatrix} 0&0&1\\0&1&0\\1&0&0\end{pmatrix}^{-1},\]
we get that
\begin{equation}\label{eq:vertical_estimates}
  d( c(a_1,a_2,a_3),c(a'_1,a'_2,a'_3)) \leq 14 L^2  e^{-(\frac{1}{2}-2s)a_1}\textrm{ if $a_1=a'_1$ and }a_2,a'_2 \leq 1.\end{equation}
In particular if $c_r= c(r,0,-r)$ and $1 \leq r_1 \leq r_2 \leq r_1+1$,
\begin{eqnarray*}  d(c_{r_2},c_{r_1}) &\leq& d(c_{r_2},c(r_2,r_1-r_2,-r_1)) + d(c(r_2,r_1-r_2,-r_1), c_{r_1})\\ & \leq & 14 L^2 (e^{-(\frac{1}{2}-2s)r_2} + e^{-(\frac{1}{2}-2s)r_2}).\end{eqnarray*}
This implies (since $\sum_{k \geq 0} e^{-(\frac{1}{2}-2s)k} \leq \frac{3}{1-4s}$) that for every $r,r' \geq 1$, \[d(c_r,c_{r'}) \leq \frac{42}{1-4s} L^2 \max(e^{-(\frac{1}{2}-2s)r},e^{-(\frac{1}{2}-2s)r'}).\] It follows easily from the above estimates that
\[ d(c(a_1,a_2,a_3),c(a'_1,a'_2,a'_3)) \leq \frac{70}{1-4s}L^2 \max( e^{-(\frac{1}{2}-2s)\max(a_1,-a_3)}, e^{-(\frac{1}{2}-2s)\max(a_1,-a_3)}),\]
which is exactly \eqref{eq:mg_Cauchy}. The previous computations are best understood on a picture (see Figure \ref{figure:zigzag})~: \eqref{eq:horizontal_estimates} expresses that $c$ is almost constant on lines of slope $-\frac 1 2$ in the region $s\geq -1$, whereas \eqref{eq:vertical_estimates} expresses that $c$ is almost constant on vertical lines in the region $s\leq 0$. These estimates are combined by the zig-zag path in Figure \ref{figure:zigzag}.

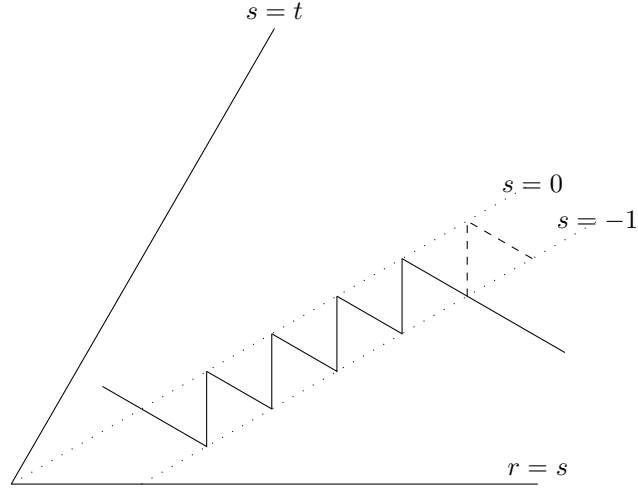
\begin{figure}
  \center
\begin{tikzpicture}[scale=1]
\draw (0,0)--(0:7) node[above] {$r=s$};
\draw (0,0)--(60:7)  node[above] {$s=t$};
\draw[shift=(-30:1), loosely dotted] (30:1)--(30:8) node {$s=-1$};
\draw[loosely dotted] (0,0)--(30:8) node {$s=0$};
 \foreach \x in {0,1,2,3}
  {
\draw[shift=(30:2+\x)] (0,0)--(-30:1)--(30:1);
  }
\draw[shift=(30:2)] (0,0)--(150:.6);
\draw[shift=(30:2+4)] (0,0)--(-30:2.5);
\draw[shift=(30:2+5), dashed] (-30:1)--(0,0)--(-90:1);
\end{tikzpicture}
\caption{The zig-zag path in the Weyl chamber $\Lambda$.} \label{figure:zigzag}
\end{figure}

We now move to the proof of \eqref{eq:nugHaar_close_to_mg}. We start by a general lemma, valid for any pair of compact groups $U \subset K$.
\begin{lemma} Every admissible probability measure $\nu$ on $K$ can be written as $\nu_1  \mu$ for admissible probability measures $\nu_1$ on $K$ and $\mu$ on $U$.
\end{lemma}
\begin{proof}
By assumption, the Radon-Nikodym derivative $d\nu/dk$ of $\nu$ is positive and is a coefficient of a finite dimensional representation $V$ of $K$. Denote by $C_V$ the finite dimensional space of real-valued matrix coefficients of $V$, equipped (say) with the $L_\infty(K)$-norm. Let $\mu_n$ be a sequence of admissible probability measures on $U$ converging weak-* to $\delta_e$. Then $T_n \colon f \in C_V \mapsto f \ast \mu_n \in C_V$ converges pointwise to the identity. Since $C_V$ has finite dimension, for $n$ large enough this linear map is invertible and there is a sequence $f_n \in C_V$ converging to $d\nu/dk$ such that $T_n f_n = d\nu/dk$. Since $d\nu/dk$ is positive, so is $f_n$ for $n$ large enough. In other words, $\nu_1=f_n dk$ is a probability measure such that $\nu_1  \mu_n=\mu$, as requested. 
\end{proof}

Let us fix $\nu$ an admissible and central probability measure on $K$. Let $\nu = \nu_1 \mu$ be a decomposition given by the previous lemma. Since $D_\alpha$ commutes with every element of $U$, one can write for $\alpha>0$ and $\delta \in [-1,1]$
\[ \nu \delta_{D_\alpha k_\delta D_\alpha} \lambda_K = \nu_1 \delta_{D_\alpha} \mu \delta_{k_\delta} \lambda_U \delta_{D_\alpha} \lambda_K.\]
By the convexity and lower-semicontinuity of $d$, the distance
\[d(\nu \delta_{D_\alpha k_\delta D_\alpha} \lambda_K, \nu \delta_{D_\alpha k_0 D_\alpha} \lambda_K)\]
is therefore bounded by
\[ \sup_{k,k' \in K} d(\delta_{kD_\alpha} \mu \delta_{k_\delta} \lambda_U \delta_{D_\alpha k'},\delta_{kD_\alpha} \mu \delta_{k_0} \lambda_U \delta_{D_\alpha k'}).\]
By combining this inequality with the last point in Lemma \ref{lem:properties_of_d}, \eqref{eq:twisted_Tdelta_Hoelder}, \eqref{eq:KAK_dec_for_k_0} and Lemma \ref{lem:spheres_distorted}, we get a constant $C(\nu)$ such that for every $\alpha$ and $r\in [\alpha,4\alpha]$
\[ d(\nu \delta_{u_{r,\alpha} D(r,2\alpha -r,-2\alpha)}\lambda_K, \nu \delta_{D(\alpha,\alpha,-2\alpha)}\lambda_K) \leq C(\nu) L^2 e^{\frac{r}{2}-(2-4s)\alpha}.\]
In particular and as for \eqref{eq:horizontal_estimates}, if $a=(a_1,a_2,a_3)\in\Lambda$ satisfies $a_2 \geq 0$ and $a_3=-2\alpha$, there is $u_a \in \widetilde U$ such that
\[ d(\nu \delta_{u_a D(a_1,a_2,a_3)} \lambda_K,\nu \delta_{D(\alpha,\alpha,-2\alpha)} \lambda_K )\leq C(\nu) L^2  e^{(1-4s)\alpha}.\]
Let us apply the preceding to the distance $d'(m,m')=d(\delta_{u u_{a}^{-1}}m,\delta_{u u_{a}^{-1}}m')$ for some $u \in \widetilde U$, which satisfies the same assumptions as $d$. Note that since $\nu$ is central and $D(\alpha,\alpha,-2\alpha)$ commutes with $u u_{a}^{-1}$, we have
\[ \delta_{u u_{a}^{-1}} \nu \delta_{u_{a} D(a_1,a_2,a_3)}\lambda_K = \delta_{u} \nu \delta_{D(a_1,a_2,a_3)}\lambda_K\]
and
\[ \delta_{u u_a^{-1}} \nu \delta_{D(\alpha,\alpha,-2\alpha)}\lambda_K = \nu \delta_{D(\alpha,\alpha,-2\alpha)}\lambda_K.\]
Therefore we obtain
\begin{equation}\label{eq:horizontal_estimate_twisted} d(\delta_{u} \nu \delta_{D(a_1,a_2,a_3)}\lambda_K,\nu \delta_{D(\alpha,\alpha,-2\alpha)}\lambda_K) \leq C(\nu) L^2 e^{(1-4s)\alpha}.\end{equation}
In particular, by the triangle inequality
\[d(\delta_{u} \nu \delta_{D(a_1,a_2,a_3)}\lambda_K,\nu \delta_{D(a_1,a_2,a_3)}\lambda_K) \leq 2C(\nu) L^2 e^{(\frac 1 2-2s)a_3}\]
for every $u \in \widetilde U$ and $(a_1,a_2,a_3)\in \Lambda$ with $a_2 \geq 0$. Similarly by applying the Cartan automorphism $\rho$ we obtain
\[ d(\delta_{u} \nu \delta_{D(a_1,a_2,a_3)}\lambda_K,\nu \delta_{D(a_1,a_2,a_3)}\lambda_K) \leq 2C(\nu) L^2 e^{-(\frac 1 2-2s)a_1}\]
for every $u \in U$ and $(a_1,a_2,a_3)\in \Lambda$ with $a_2 \leq 0$. Indeed, $\rho$ exchanges $U$ and $\widetilde U$ and preserves $\nu$.

Consider for a moment the particular case $a_2=0$. In that situation both estimates can be applied, and give that
\[ d(\delta_{u} \nu \delta_{D(a,0,-a)}\lambda_K,\nu \delta_{D(a,0,-a)}\lambda_K) \leq 2C(\nu) L^2 e^{-(\frac 1 2-2s)a}\]
for every $u \in U \cup \widetilde U$. But every element of $K$ can be written as a product of $\leq 3$ elements of $U \cup \widetilde U$, so the preceding inequality implies
\[ d(\delta_{k} \nu \delta_{D(a,0,-a)}\lambda_K,\nu \delta_{D(a,0,-a)}\lambda_K) \leq 6C(\nu) L^2 e^{-(\frac 1 2-2s)a}\]
for every $k \in K$. If we average with respect to $K$ (and use one last time the convexity and lower-semicontinuity of $d$) we obtain
\[ d(\lambda_K \delta_{D(a,0,-a)}\lambda_K,\nu \delta_{D(a,0,-a)}\lambda_K) \leq 6 C(\nu) L^2 e^{-(\frac 1 2 - 2s) a}.\]
By \eqref{eq:horizontal_estimate_twisted} we get
\[ d(\lambda_K \delta_{D(\alpha,\alpha,-2\alpha)}\lambda_K,\nu \delta_{D(\alpha,\alpha,-2\alpha)}\lambda_K) \leq 8 C(\nu) L^2 e^{-(1-4s)\alpha}.\]
By \eqref{eq:horizontal_estimate_twisted} again this implies that 
\[ d(\lambda_K \delta_{D(a_1,a_2,a_3)}\lambda_K,\nu \delta_{D(a_1,a_2,a_3)}\lambda_K) \leq 10 C(\nu) L^2 e^{(\frac 1 2-2s)a_3}\]
for every $(a_1,a_2,a_3)\in\Lambda$ with $a_2 \geq 0$. By symmetry (\emph{i.e.} by conjugating by the Cartan automorphism) we also get
\[ d(\lambda_K \delta_{D(a_1,a_2,a_3)}\lambda_K,\nu \delta_{D(a_1,a_2,a_3)}\lambda_K) \leq 10 C(\nu) L^2 e^{-(\frac 1 2-2s)a_1}\]
for every $(a_1,a_2,a_3)\in\Lambda$ with $a_2 \leq 0$. To summarize, we have proven that
\[ d(\lambda_K \delta_{g}\lambda_K,\nu \delta_{g}\lambda_K) \leq 10 C(\nu) L^2 e^{-(\frac 1 2-2s)\ell(g)}\]
for every $g$ of the form $D(a_1,a_2,a_3)$. Considering the $KAK$ decomposition, we obtain the validity of the preceding inequality for $g \in G$ be arbitrary. This concludes the proof of \eqref{eq:nugHaar_close_to_mg} and therefore of Theorem \ref{thm:main} for $\mathrm{SL}_3(\R)$.

\begin{rem}\label{rem:proof_SL3R_Banach} The only place in the proof where we used in an essential way that $X_1$ is a Hilbert was in the conclusion (\ref{item:local_estimate}) of Lemma \ref{lem:properties_of_d}, which allowed us to exploit \eqref{eq:T_delta_Hoelder} and \eqref{eq:twisted_Tdelta_Hoelder}. However, for Banach spaces we have Remark \ref{rem:properties_of_d_Banach}, and the rest of the above proof shows that $(\mathrm{SL}_3(\R),\ell)$ satisfies $(*_{\mathcal E})$ provided that there is $\alpha>0$ and $C>0$ such that for every $\delta \in [-1,1]$ and $X \in \mathcal E$,
  \[\|\iint_{U \times U} \lambda(u k_\delta u')-\lambda(u k_0 u') du du'\|_{B(L_2(K;X))} \leq C |\delta|^{\frac \alpha 2}\]
  and
  \[\|\iint_{U \times U} \lambda(u k_\delta u')-\lambda(u k_0 u') d\mu_1(u) d\mu_2(u')\|_{B(L_2(K;X))} \leq C(\mu_1,\mu_2) |\delta|^{\frac \alpha 2}.\]
  The first equality is equivalent to \eqref{eq:Tdelta}, and the second actually follows from the first, see \cite[Proposition 2.1]{dlMdlSAIF}. This proves Theorem \ref{thm:main_Banach_valued_real} for the group $G=\mathrm{SL}_3(\R)$.
\end{rem}

\subsection{Case of $\mathrm{SL}_3(\F)$}
The proof of Theorem \ref{thm:main} for $G=\mathrm{SL}_3(\F)$ and $\F=\Q_p$ or $\mathbf F_p(\!(t)\!)$ for some prime number $p$ is essentially the same as for $\F=\R$.

We only give a rapid overview of the small adjustements one has to make. In that case the maximal compact subgroup $K \subset G$ is $\mathrm{SL}_3(\cO)$ if $\cO$ is the ring of units of $\F$ (namely $\cO=\Z_p$ or $\mathbf F_p[\![t]\!]$). It is more natural to prove the theorem for the length function $\ell(g) = \max(\log \|g\|,\log \|g^{-1}\|)$, where $\|\cdot\|$ is the norm induced from the natural $K$-invariant norm $\|(s_1,s_2,s_3)\|=\max(|s_1|,|s_2|,|s_3|)$ on $F^3$ (where $|s|$ is the standard absolute value on $F$, \emph{i.e.} the amount by which the Haar measure on $(\F,+)$ is scaled under the multiplication by $s$). With this normalization, as for the case $F=\R$, Theorem \ref{thm:main} holds with any $s<\frac 1 4$, $t= \frac 1 2 -2s$, $C=\frac{C_0 p }{1-4s}$ for a constant universal constant $C_0$ (independant from $p$) and for any $K$-biinvariant sequence $m_n$ supported in $\{g \in G |\ell(g)=n\}$.

Indeed, Lemma \ref{lem:properties_of_d} holds in this setting, and all amounts to proving Proposition \ref{prop:local_SL3}. For that, one defines the subgroups $U,\widetilde U$ by the same formulas as for the real case, but in that case they are both isomorphic to $\mathrm{GL}_2(\cO)$. The matrix $k_\delta \in K$ is defined for every $\delta \in \cO$ by the formula
\[ k_\delta = \begin{pmatrix} \delta & -1 & 0 \\ 1 & 0 & 0\\ 0&0&1 \end{pmatrix}\]
Both formulas \eqref{eq:T_delta_Hoelder} and \eqref{eq:twisted_Tdelta_Hoelder} hold in this setting. This can be derived from \cite{MR2423763}, see also the more general Proposition \ref{prop:estimates_for_T_delta_Banach_nonarchimedean} below.

The Weyl chamber is now replaced by its discretized version $\Lambda = \{(a_1,a_2,a_3) \in \Z^3| a_1 \geq a_1 \geq a_3, a_1+a_2+a_3=0\}$, which still parametrizes the $K$-double cosets by the matrices
\[ D(\alpha,\beta,\gamma) = \begin{pmatrix} \e^\alpha & 0&0\\ 0&\e^{\beta} &0\\ 0&0&\e^{\gamma}\end{pmatrix},\]
where $\e$ denotes the inverse of a uniformizer in $\cO$. To fix ideas, $\e=p^{-1}$ if $\F=\Q_p$ and $\e= t^{-1}$ if $\F=\mathbf F_p(\!(t)\!)$.

Lemma \ref{lem:spheres_distorted} is replaced by its formal analogue
\begin{lemma} Let $\alpha \in \N$. For every integer $r \in [\alpha,4\alpha]$ there are $\delta \in \cO$ and $u_{r,\alpha},u'_{r,\alpha} \in \widetilde U$ such that $|\delta|\leq |\e|^{r-4\alpha} \leq 1$ and
  \[D_\alpha k_\delta D_\alpha = u_{r,\alpha} D(r,2\alpha-r,-2\alpha)u'_{r,\alpha}.\]
\end{lemma}
The proof is similar and actually even simpler than in the archimedean case, because when $F$ is non-archimedean the operator norm of a matrix $g \in \mathrm{SL}_3(F)$ is simply $\|g\|=\max_{i,j} |g_{i,j}|$. Therefore one may take $\delta=\e^{r-4\alpha}$.

Proposition \ref{prop:local_SL3} is deduced from \eqref{eq:T_delta_Hoelder}, \eqref{eq:twisted_Tdelta_Hoelder} and the preceding lemma in the same way as in the real case. There is just one difficulty that occurs from the discreteness of $\Lambda$. Indeed, since $\alpha$ has to be an integer in the preceding lemma, one obtains \eqref{eq:Hoelder_continuity_c} only when $a_3=a'_3 = -2\alpha$ is even. A way to obtain the same inequality also when $a_3= a'_3$ is odd (say equal to
$1-2\alpha$ for an integer $\alpha$) is to apply the same reasoning (by replacing $D_\alpha k_\delta D_\alpha$ by $D_\alpha k_\delta D_{\alpha-1}$) to the new distance $d'(m,m') = d(u_0(m),u_0(m'))$ for $u_0$ is the (non-inner) automorphism of $G$ preserving $U$ and given by 
\begin{equation}\label{eq:theta_0} u_0(g) = \begin{pmatrix} \e^{-1} & 0 & 0 \\ 0& 1 &0\\ 0&0&1\end{pmatrix} g\begin{pmatrix} \e & 0 & 0 \\ 0& 1 &0\\ 0&0&1\end{pmatrix}.\end{equation}
Indeed, one checks easily that $u_0(D_\alpha k_\delta D_{\alpha-1})) \in \widetilde U D(a_1,a_2,a_3) \widetilde U$ if $\delta = e^{2a_3-a_1}$, and to $D(\alpha,\alpha-1,1-2\alpha) \widetilde U$ if $\delta=0$, and this leads to the validity of \eqref{eq:Hoelder_continuity_c} also when $a_3= a'_3$ is odd, at the cost of replacing $L$ by $L |\e|^s$. The reason for this is that, since $u_0$ does not preserve the length but only satisfies $|\ell(g) - \ell(u_0(g))| \leq \log |\e|$, Lemma \ref{lem:properties_of_d} holds for $d'$ with $L$ replaced by $L|\e|^s$. The same adjustment has to be made to obtain \eqref{eq:nugHaar_close_to_mg}.

Another way to fix this parity issue is to work from the beginning with $\mathrm{PGL}_3(\F)$ where the automorphism $u_0$ becomes inner, as in \cite{MR2423763}.

\begin{rem}\label{rem:proof_SL3F_Banach} When $X_1$ is a Banach space, Remark \ref{rem:proof_SL3R_Banach} holds similarly, except that in that case we know exactly for which Banach spaces there exists $\alpha,C>0$ such that for every $\delta \in \cO$
\[\|\iint_{U \times U} \lambda(u k_\delta u')-\lambda(u k_0 u') du du'\|_{B(L_2(K;X))} \leq C |\delta|^{\frac \alpha 2}.\]
These are exactly the Banach spaces of nontrivial Rademacher type, see \cite{MR2574023} or Proposition \ref{prop:estimates_for_T_delta_Banach_nonarchimedean} below. This proves Theorem \ref{thm:main_Banach_valued_nonarch} for the group $G=\mathrm{SL}_3(\F)$.
\end{rem}

We are left to prove the following proposition, which is a variant of \cite[Lemme 4.4]{MR2574023}.
\begin{prop}\label{prop:estimates_for_T_delta_Banach_nonarchimedean} Let $\mathcal E$ be a class of Banach spaces with nontrivial type. There is $\alpha>0$, and for every admissible probability measures $\mu_1,\mu_2$ on $\widetilde U$, a constant $C(\mu_1,\mu_2)$ such that
\[\|\iint_{U \times U} \lambda(u k_\delta u')-\lambda(u k_0 u') d\mu_1(u) d\mu_2(u')\|_{B(L_2(K;X))} \leq C(\mu_1,\mu_2) |\delta|^{\frac \alpha 2}.\]  
\end{prop}
The starting point is the following consequence of the Hausdorff-Young inequality of Bourgain \cite{MR675400}, see \cite[Corollaire 2.2]{MR2574023}: there is $\varepsilon>0$ and $C>0$ such that for every finite abelian group $G$, every $X \in \mathcal E$ and every $f \colon G \to X$,
\begin{equation}\label{eq:Hausdorff-Young}
  \left(\sum_{\chi \in \hat G} \| \mathbb E_{s \in G} \chi(s) f(s)\|^2 \right)^{\frac 1 2} \leq  C (\#G)^{-\varepsilon}   \left(\sum_{s \in G} \| f(s)\|^2 \right)^{\frac 1 2}.
\end{equation}
Let us fix $\alpha \in (0,\varepsilon)$. We shall prove Proposition \ref{prop:estimates_for_T_delta_Banach_nonarchimedean} for this value of $\alpha$, by applying the preceding to the additive group of the residue rings $\cO_n:= \cO/\e^{-n}\cO$ for different values of $n$, in which case $C (\#G)^{-\varepsilon} = Cp^{-\varepsilon n}$. We start with the following consequence, which follows rather easily from Lemma 3.2 in \cite{MR2574023}.
\begin{lemma}\label{lemma:consequence_of_Bourgain} There is a constant $C'$ such that for every integer $h \geq 1$, every nontrivial character $\chi$ of the additive group of $\cO_h$ and every integer $n \geq h$, the operator $S_{n,\chi} \in B(\ell_2(\cO_n \times \cO_n))$ defined by
  \[ S_{n,\chi} f(y,t) = \mathbb E_{x \in \cO_n,z \in \cO_h} \chi(z) f(x,t+\e^{h-n}z+xy)\]
  satisfies
  \[ \|S_{n,\chi}\|_{B(\ell_2(\cO_n \times \cO_n;X))} \leq C' p^{-\alpha (n-h)}.\]
  for every $X \in \mathcal E$.
\end{lemma}
\begin{proof} Define $h_0$ as the smallest integer satisfying $(C p^{-\varepsilon h_0}) \leq p^{-\alpha h_0}$. We consider three cases.

Case 1: $\chi$ is non-degenerate, that is $\chi$ is not trivial on the subgroup $\e^{-h+1}\cO/\e^{-h}\cO \subset \cO_h$, and $h \geq h_0$. In that case, Lemma 3.2 in \cite{MR2574023} applies, and together with \eqref{eq:Hausdorff-Young}, shows that
\[ \|S_{n,\chi}\|_{B(\ell_2(\cO_n \times \cO_n;X))} \leq (C p^{-\varepsilon h})^{\frac{n}{h}-1}
  \leq  p^{-\alpha (n-h)},\]
  where the last inequality holds because $(C p^{-\varepsilon h}) \leq p^{-\alpha h}$ (we assumed $h \geq h_0$).

  Case 2: $1 \leq h < h_0$ and $\chi$ is non-degenerate. The homomorphism $z \in \cO_h \mapsto \e^{h-h_0}z \in \cO_{h_0}$ induces a surjective $q \colon \hat \cO_{h_0} \to \hat \cO_h$, and for every $\widetilde z \in \cO_{h_0}$, the orthogonality of characters implies that 
  \[ \sum_{\widetilde \chi \in q^{-1}(\chi)} \widetilde \chi(\widetilde z) = \left\{ \begin{array}{cc}p^{h_0-h}\chi(z)&\textrm{if }\widetilde z=\e^{h-h_0} z \in \e^{h-h_0} \cO_h\\0&\textrm{otherwise.}\end{array}\right.\]
As a consequence we have $S_{n,\chi} = \sum_{\widetilde \chi \in q^{-1}(\chi)} S_{n,\widetilde \chi}$, and moreover every character in $q^{-1}(\chi)$ is non-degenerate. Taking into account Case 1, we obtain
  \[ \|S_{n,\chi}\|_{B(\ell_2(\cO_n \times \cO_n;X))} \leq p^{h_0-h} p^{-\alpha (n-h_0)} = p^{(1+\alpha)(h_0-h)}p^{-\alpha(n-h)} \textrm{ for every }X\in \mathcal E.\]

  To summarize, we have proven the lemma with $C'=p^{(1+\alpha)(h_0-1)}$ under the additional restriction that $\chi$ is non-degenerate.

  In the general case, let $d \geq 1$ be the largest integer such that $\chi$ is trivial on $\e^{-d} \cO/e^{-h}\cO$. Then $\chi$ induces a non-degenerate character $\chi'$ of $\cO_d$, and one checks that $S_{n,\chi} = i\circ S_{n-h+d), \chi'} \circ P$ where $i\colon \ell_2(\cO_{n-h+d}\times \cO_{n-h+d}) \to \ell_2(\cO_n\times\cO_n)$ is the natural isometric embedding and $P$ is the orthogonal projection.
  This shows
  \[ \|S_{n,\chi}\otimes \mathrm{id}_X\|\leq  \|S_{n-h+d,\chi'}\mathrm{id}_X\| \leq C' p^{-\alpha(n-h+d-d)} = C'p^{-\alpha(n-h)}\]
and proves the lemma.
  \end{proof}
Now for an integer $n$ and $\delta \in \cO_n$, we define the operator $S_{n,\delta}$ on $\ell_2(\cO_n \times \cO_n)$ by
\[ S_{n,\delta} f(y,t) = \mathbb E_{x \in \cO_n} f(x,t+\delta+xy).\]
We deduce
\begin{lemma}\label{lemma:inequality_for_Sndelta} For every integer $h$, there is a constant $C(h)$ such that for every $n \geq h$ and $\delta,\delta' \in \e^{h-n}\cO_h$,
  \[ \|S_{n,\delta} - S_{n,\delta'}\|_{B(\ell_2(\cO_n\times\cO_n;X))} \leq C(h) p^{-\alpha(n-h)}\]
  for every $X \in \mathcal E$.
\end{lemma}
\begin{proof} 
  Write $\delta=\e^{h-n}a$ and $\delta'=\e^{h-n} b$ with $a,b \in \cO_{n-h}$, and consider the function $\varphi =p^h(\delta_{a} - \delta_{b}) \colon \cO_h \to \R$. We can decompose $\varphi$ in the basis of characters $\varphi = \sum_{\chi \in \hat{\cO_h}} t_\chi \chi$. Since $\varphi$ has mean $0$, the trivial character does not appear in this decomposition, and it follows from the definitions that
  \[ S_{n,\delta}-S_{n,\delta'} = \sum_{\chi \in \hat{\cO_h}} t_\chi S_{n,\chi}.\]
  Lemma \ref{lemma:consequence_of_Bourgain} implies that
  \[  \|S_{n,\delta} - S_{n,\delta'}\|_{B(\ell_2(\cO_n\times\cO_n;X))} \leq C' p^{-\alpha(n-h)} \sum_\chi |t_\chi|.\]
  The lemma follows with $C(h) = 2C'p^h$ because $t_\chi = \chi(a) - \chi(b)$ has modulus $\leq 2$.
\end{proof}
When $\mathcal E$ is made of Hilbert spaces, a direct diagonalization of the operators $S_{n,\delta}$ show that Lemma \ref{lemma:inequality_for_Sndelta} holds with a constant $C(h)$ independant from $h$. We believe that this should hold for every class $\mathcal E$ of nontrivial type, but we could not prove it.

We now move to the proof of Proposition \ref{prop:estimates_for_T_delta_Banach_nonarchimedean}.
\begin{proof}[Proof of Proposition \ref{prop:estimates_for_T_delta_Banach_nonarchimedean}] By the very same argument as in the proof of \cite[Proposition 2.1]{dlMdlSAIF} (suitably adapted to replace Lie groups by totally disconnected groups), we could restrict to the case when $\mu_1=\mu_2$ is the Haar measure on $U$. But it does not require much more effort to provide the argument in the general case, we therefore do so. Let $\mu_1,\mu_2$ be two admissible probability measures on $U$, and for $k \in K$ denote by $A_k, B_k,C_k$ the following operators on $L_2(K)$:
\[ A_k=\iint_{U \times U} \lambda(u k u') d\mu_1(u) d\mu_2(u),\]
\[ B_k=\int_{U} \lambda(u k) d\mu_1(u),\]
\[ C_k=\iint_{U} \lambda(k u')d\mu_2(u),\]
so that $B_k C_{k'} = A_{kk'}$.

For every integer $j$, denote by $U_j$ the kernel of the reduction morphism $U \to \SL_3(\cO_j)$. Since $(U_j)_j$ forms a basis of neighbourhoods of the identity in $U$, every finite dimensional representation of $U$ is trivial on $U_j$ for all $j$ large enough, and therefore every admissible probability measure on $U$ is left and right-invariant under $U_j$ for all $j$ large enough. Fix $j$ such that this invariance holds for $\mu_1$ and $\mu_2$. So we have that $A_k = A_{uku'}$ for every $k \in K$ and $u,u' \in U_j$.

Let $n \geq j$. Denote by $x\in \cO_n \mapsto \sect x \in \cO$ any section. For $a,b,x,y \in \cO_n$ we define matrices in $K$ 
\[ \alpha(a,b) = \begin{pmatrix} 1 & -\e^{-j}\sect{a}& -\e^{-2j}\sect{b} \\ 0 & 0 &1\\ 0&-1&0 \end{pmatrix}\ ,\ \beta(x,y) = \begin{pmatrix} \e^{-2j}\sect{y} & -1 & 0 \\  \e^{-j}\sect{x} & 0 & -1 \\ 1&0&0 \end{pmatrix}.\]
Then 
\[ \alpha(a,b) \beta(x,y) =  \begin{pmatrix} \e^{-2j}(\sect{y}-\sect{a}\sect{x}-\sect{b}) & -1 & \e^{-j}\sect{a} \\ 1 & 0 &0 \\ -\e^{-j} \sect{x}&0&1 \end{pmatrix}.\]
If $\delta \in \cO$ is such that $|\delta| \geq p^{j-n}$, then for every $a,b,x,y \in \cO_n$ such that $y-ax-b=\delta+\e^{-n}\cO$ we have $\sect{y}-\sect{a}\sect{x}-\sect{b} \in \delta + \e^{-n}\cO$ and $\omega = \frac{\sect{y}-\sect{a}\sect{x}-\sect{b}}{\delta} \in 1+\e^{-j}\cO$. We have that $\alpha(a,b) \beta(x,y) \in U_j k_{\e^{-2j}\delta} U_j$ as the explicit complutation shows:
\[ \begin{pmatrix} \omega^{-1} & 0 &0\\0&1&0\\0&\e^{-j} \omega \sect{x}&\omega \end{pmatrix} \alpha(a,b) \beta(x,y) \begin{pmatrix} 1 & 0 &0\\0&\omega&\e^{-j} \omega^{-1} \sect{a}\\0&0&\omega^{-1} \end{pmatrix}=  k_{\e^{-2j}\delta}.\]
In particular, $A_{ \alpha(a,b) \beta(x,y)} = A_{k_{\e^{-2j}\delta}}$.

Consider unit vectors $\xi \in L_2(K;X)$ and $\eta \in L_2(K;X)^*$. Define $f \colon \cO_n \times \cO_n \to L_2(K;X)$ and $g \colon \cO_n \times \cO_n \to L_2(K;X)^*$ by $f(x,y) = B_{\beta(x,y)} \xi$ and $g(a,b) = C_{\alpha(a,b)}^* \eta$. Then the norm of $f$ in $\ell_2(\cO_n \times \cO_n;L_2(K;X))$ is less than $p^n$ and similarly for $g$. Moreover
\[ \langle f(x,y),g(a,b)\rangle = \langle C_{\alpha(a,b)} B_{\beta(x,y)}\xi,\eta\rangle = \langle A_{\alpha(a,b) \beta(x,y)} \xi,\eta \rangle .\]
For $\delta$ as above, denote $\overline{\delta}$ its image in $\cO_n$ and compute
\begin{eqnarray*} \langle S_{n,\overline{\delta}}f,g\rangle &=& \sum_{a,b \in \cO_n} \mathbb E_{x \in \cO_n} \langle f(x,b+\overline{\delta}+ax),g(a,b)\rangle\\
  & = & \sum_{a,b \in \cO_n} \mathbb E_{x \in \cO_n} \langle A_{ k_{\e^{-2j}\delta}}\xi,\eta\rangle\\
  &=& p^{2n} \langle A_{k_{\e^{-2j} \delta}}\xi,\eta\rangle.\end{eqnarray*}
We deduce that, for $\delta,\delta'$ of absolute value $\geq p^{j-n}$, 
\[ \langle (A_{k_{\e^{-2j}\delta}}-A_{k_{\e^{-2j}\delta'}})\xi,\eta\rangle = \langle (S_{n,\overline{\delta}} - S_{n,\overline{\delta'}})f/p^n,g/p^n \rangle.\]
Taking the supremum over $\xi,\eta$ we obtain for every Banach space $X$,
\[ \| A_{k_{\e^{-2j}\delta}}-A_{k_{\e^{-2j}\delta'}} \|_{B(L_2(K;X))} \leq \| S_{n,\overline{\delta}} - S_{n,\overline{\delta'}}\|_{B(\ell_2(\cO_n \times \cO_n;L_2(K;X)))}.\]
By Fubini this last quantity is equal to $\| S_{n,\overline{\delta}} - S_{n,\overline{\delta'}}\|_{B(\ell_2(\cO_n \times \cO_n;X))}$. From now on we assume $X \in \mathcal E$. We deduce from Lemma \ref{lemma:inequality_for_Sndelta} that if $n \geq h \geq j$ and $|\delta|,|\delta'|$ belong to $[p^{j-n},p^{h-n}]$ then
\[ \| A_{k_{\e^{-2j}\delta}}-A_{k_{\e^{-2j}\delta'}} \|_{B(L_2(K;X))} \leq C(h) p^{-\alpha(n-h)}.\]
In particular, if $|\delta|=p^{1+j-n}$ and $|\delta'|=p^{j-n}$, we can take $h=j+1$ and obtain
\[ \| A_{k_{\e^{-2j}\delta}}-A_{k_{\e^{-2j}\delta'}} \|_{B(L_2(K;X))} \leq C(j+1) |\delta|^\alpha.\]
Making $n$ vary, we see that the previous inequality holds for every $\delta,\delta'\in \cO$ satisfying $|\delta|/|\delta'|=p$. Now take $\delta,\delta' \in \cO$ with $0<|\delta'| \leq |\delta|$. Denote $|\delta|=p^{-m}$ and $|\delta'|=p^{-m'}$, so $m \leq m'$. Pick a sequence $\delta=\delta_0,\delta_1,\dots,\delta_{m-m'}=\delta'$ such that $|\delta_i|=p^{-(m+i)}$. Then we have
\begin{multline*} \| A_{k_{\e^{-2j}\delta}}-A_{k_{\e^{-2j}\delta'}} \|_{B(L_2(K;X))} \leq \sum_{i=1}^{m-m'} \| A_{k_{\e^{-2j}\delta_{i-1}}}-A_{k_{\e^{-2j}\delta_i}} \|_{B(L_2(K;X))} \\\leq \sum_{i=1}^{m-m'} C(j+1) p^{-\alpha(m+i-1)} \leq \frac{C(j+1)}{1-p^{-\alpha}} |\delta|^\alpha.\end{multline*}
Making $\delta' \to 0$ we obtain
\[ \| A_{k_{\e^{-2j}\delta}}-A_{k_{0}} \|_{B(L_2(K;X))} \leq \frac{C(j+1)}{1-p^{-\alpha}} |\delta|^\alpha.\]
This concludes the proof of Proposition \ref{prop:estimates_for_T_delta_Banach_nonarchimedean}.
\end{proof}
\subsection{Case of $\mathrm{Sp}_4(\F)$} The case of $\mathrm{Sp}_4(\F)$ for $\F=\R,\Q_p$ or $\mathbf F_p(\!(t)\!)$ proceeds in the same way as for $\mathrm{SL}_3(\F)$~: Lemma \ref{lem:properties_of_d} holds without any change, and again all amounts to proving Proposition \ref{prop:local_SL3}. And this is achieved by adapting the known proofs of strong property (T) \cite{liao,strongTsp4} in the same way as for $\mathrm{SL}_3$. Actually, the same strategy allows us, as in \cite{strongTsp4}, to prove Theorem \ref{thm:main} for the universal cover $\widetilde{\mathrm{Sp}}_4(\R)$. The same arguments also show the conclusion of Theorem \ref{thm:main_Banach_valued_nonarch} for $\mathrm{Sp}_4(\F)$ and of Theorem \ref{thm:main_Banach_valued_real} for $\mathrm{Sp}_4(\R)$ and $\widetilde{\mathrm{Sp}}_4(\R)$. We leave the details to the reader.

\subsection{Case of $\mathrm{SL}_{3n-3}(\R)$}\label{subsection:SLbigN} A particular case of the results of the next section is that $\mathrm{SL}_N(\R)$ satisfies property (*) for every $N\geq 3$, and more generally that it satisfies $(*_{\mathcal E})$ whenever $\mathrm{SL}_3(\R)$ does. The reason is that $\mathrm{SL}_3 \subset \mathrm{SL}_n$. This does not prove Theorem \ref{thm:main_Banach_valued_real} for $\mathrm{SL}_{3n-3}(\R)$ because it is unknown whether condition \eqref{eq:Tdeltan} implies \eqref{eq:Tdelta}. However, we can inject in the general argument presented in Subsection \ref{subsection:SL3R} the exploration process of the Weyl chamber of $\mathrm{SL}_{3n-6}$ that was obtained in \cite{dlMdlSAIF} using copies of $\mathbb{S}^{n-1}$ and obtain Theorem \ref{thm:main_Banach_valued_real} for $\mathrm{SL}_{3n-3}(\R)$. 

\section{Generalization to other higher rank groups}\label{sec:all_higher_rank_groups}

We now prove Theorem \ref{thm:main} for higher rank groups, using that we already know that it holds for $\mathrm{SL}_3$, $\mathrm{Sp}_4$ and $\widetilde{\mathrm{Sp}}_4$. By Lemma \ref{lem:direct_products} we can restrict to the case of higher rank \emph{simple} groups. The arguments are close to \cite[Section 4]{MR2423763} and \cite[\S 5]{liao}. We refer to these references for the missing details.

We start with the proof for the case when $G$ is a connected simple Lie group of real rank $\geq 2$.

By the classification of higher rank Lie algebras $G$ contains a closed subgroup $H$ whose Lie algebra is $\mathfrak{sl}_3$ or $\mathfrak{sp}_2$, see \cite[Proposition I.1.6.2]{MR1090825}. This means that $H$ is isomorphic to a finite extension of $\mathrm{SL}_3(\R)$ or $\mathrm{Sp}_2(\R)$, or to $\widetilde{\mathrm{Sp}}_4(\R)$. Since, as the reader can easily check, property (*) remains true if one replaces $H$ by a finite extension or by a quotient by a finite group, we know that the theorem holds for $H$: let $s(H),t(H)>0$ and $m_n$ be a sequence of probability measures supported on $\{h \in H, \ell(h) \leq n\}$ as in the theorem.

Let $a \in H$ be the exponential of a nonzero semisimple element $X$, and $\mathfrak g = \oplus_\lambda \mathfrak{g_\lambda}$ be the decomposition as eigenspaces for $\mathrm{Ad}(\mathfrak{a})$. In this way, for $Y \in \mathfrak{g}_\lambda$, $\mathrm{ad}(a)(Y) = e^{\lambda} Y$ and $a \exp(Y) a^{-1} = \exp( e^\lambda Y)$.

We shall prove that (*) holds for $(G,\ell)$ with the parameters $s,t=t(H)$ and the sequence $m_n$ (seen as probability measures on $G$) if $s \leq s(H)$ is small enough.

Let $\pi$ be as in (*) for $(G,\ell)$. Then since $s \leq s(H)$, we know that there is $P \in B(X_0,X_2)$ such that 
\begin{equation}\label{eq:alpha_nCauchy} \|\pi(m_n)-P\| \leq C e^{-nt}.\end{equation}
Moreover, if $\mu,\nu$ are probability measures on $G$ supported in $\{g,\ell(g) \leq R\}$ and $\{g,\ell(g) \leq R'\}$ respectively then by applying the preceding to $\widetilde \pi(\cdot) = \pi(\mu \cdot \nu)$ we obtain ${}_\mu P_\nu \in B(X_1,X_2)$ such that
\begin{equation}\label{eq:alpha_nCauchy_distorted} \|\pi(\mu  m_n  \nu)-{}_\mu P_\nu\| \leq C L^2 e^{s(R+R')-nt}.\end{equation}
We have to prove that ${}_\mu P_\nu=P$ for every $\mu,\nu$. By assumption, we know that ${}_\mu P_\nu=P$ if $\mu,\nu$ are supported in $H$. More generally ${}_\mu P_\nu$ only depends on the images of $\mu$ (resp. $\nu$) in $G/H$ (resp. in $H\backslash G$). Finally as in Lemma \ref{lemma:lsc} the map $(\mu,\nu) \mapsto {}_\mu P_\nu$ is lower-semicontinuous.

\begin{lemma}\label{lem:P_invariant_by_expglambda} The following holds if $s$ is small enough. Let $\mu,\nu$ be compactly supported measures on $G$. Then for every $\lambda \neq 0$ and $Y \in \mathfrak{g}_\lambda$,
  \[ {}_{\mu  \delta_{\exp(Y)}} P_\nu = {}_\mu P_\nu.\]
\end{lemma}

\begin{proof} Since $\nu \mapsto \|{}_{\mu  \delta_{\exp(Y)}} P_\nu - {}_{\mu} P_\nu\|$ is lower semicontinuous, we can restrict the proof to the case when $\nu$ belongs to some dense subset. So we can assume that $\nu$ is absolutely continuous with respect to the Haar measure with a compactly supported and Lipschitz Radon-Nikodym derivative. By replacing $a$ by $a^{-1}$ (\emph{i.e.} $X$ by $-X$) we can also assume that $\lambda>0$. Let $c$ be an integer, the value of which will be determined at the end of the proof. From the preceding discussion, we know that ${}_{\mu  \delta_{\exp(Y)}} P_\nu = {}_{\mu  \delta_{\exp(Y)a^{cn}}} P_\nu$ and ${}_{\mu} P_\nu = {}_{\mu  \delta_{a^{cn}}} P_\nu$. Applying \eqref{eq:alpha_nCauchy_distorted} with $\mu$ replaced by $\mu  \delta_{\exp(Y) a^{cn}}$ and $\mu  \delta_{a^{cn}}$ we obtain by the triangle inequality
  \begin{multline*} \| {}_{\mu  \delta_{\exp(Y)}} P_\nu - {}_\mu P_\nu \| \leq 2CL^2 e^{s(R+R'+cn\ell(a)+\ell(\exp(Y)))-nt} \\ + \|\pi(\mu \delta_{exp(Y) a^{cn}} m_n \nu - \mu \delta_{a^{cn}} m_n \nu)\|.\end{multline*}
  The first term goes to zero as $n \to \infty$ if $s<\frac{t}{c\ell(a)}$.
  
  Let us bound the second term. Let $\alpha_n$ denote the measure $\mu \delta_{exp(Y) a^{cn}} m_n \nu - \mu \delta_{a^{cn}} m_n \nu$. Then $\alpha_n$ is supported in $\{g \in G| \ell(g) \leq R_n\}$ for some $R_n \leq C' +n(1+\ell(a))$. Therefore, $\|\pi(\alpha_n)\| \leq CL^2e^{sR_n} \|\alpha_n\|_{TV}$. But by the triangle inequality, the total variation norm $\|\alpha_n\|_{TV}$ of $\alpha_n$ is less than
  \[ \sup_{\ell(g) \leq n} \|\delta_{g^{-1}a^{-cn}\exp(Y)a^{cn}g}\nu - \nu\|_{TV}.\]
  However $a^{-cn}\exp(Y)a^{cn} = \exp(e^{-\lambda cn}Y)$ because $Y \in \mathfrak{g}_\lambda$. Let $C>0$ be such that $\mathrm{ad}(g)$ is $e^C$-Lipschitz on the $1$-neighbourhood of the identity in $G$ for every $g$ with $\ell(g) \leq 1$. For $\ell(g) \leq n$,  $\mathrm{ad}(g)$ is $e^{nC}$-Lipschitz on the $e^{-C(n-1)}$-neighbourhood of the identity, so that in particular is $c$ is large enough (namely such that $C-\lambda c<0$) we have that $g^{-1}a^{-cn}\exp(Y)a^{cn}g$ is at distance $O(e^{(C-\lambda c)n})$ from the identity. Remembering that  $\nu$ has a compactly supported and Lipschitz Radon-Nikodym derivative, we obtain
\[\sup_{\ell(g) \leq n} \|\delta_{g^{-1}a^{-cn}\exp(Y)a^{cn}g}\nu - \nu\|_{TV} = O(e^{(C-\lambda c)n}).\]
One deduces that $\|\pi(\alpha_n)\|$ goes to zero if $e^{C+s(1+\ell(a))-\lambda c}<1$. To conclude, if $c$ is chosen so that $\lambda c > C$ for every $\lambda \neq 0$ in the sectrum of $ad(a)$, then $\| {}_{\mu  \delta_{\exp(Y)}} P_\nu - {}_\mu P_\nu \|=0$ provided that $s \leq \min (\frac{\lambda c - C}{1+\ell(a)},\frac{t}{c\ell(a)})$.
\end{proof}

We can now conclude the proof of the theorem, for the value of $s \leq s(H)$ given by the preceding lemma. Clearly, the set of elements of $G$ such that ${}_{\mu  \delta_{g}} P_\nu = {}_\mu P_\nu$ for every compactly supported measures $\mu,\nu$ on $G$ is a group. Lemma \ref{lem:P_invariant_by_expglambda} shows that this group contains the group generated by $\cup_{\lambda \neq 0} \exp(\mathfrak{g}_\lambda)$, which is the whole group $G$. Similarly we have that ${}_\mu P_{\delta_g \nu}={}_\mu P_\nu$ for every $g \in G$. In particular we have ${}_{\delta_{g}} P_{\delta_{g'}}= P_{\delta_{g'}} = P$ for every $g,g' \in G$. This proves the theorem.



Consider now a non-archimedean local field $\F$ and an almost $\F$-simple algebraic group $G$ with $\F$-split rank $\geq 2$. We can assume that $F=\Q_p$ or $\mathbf F_p(\!(t)\!)$ for some prime number $p$. Indeed, if $F$ is a finite extension of $F'$ then $G$ is isomorphic to an almost $F'$-simple algebraic group with $F'$-split rank $\geq 2$. Moreover, replacing $G$ by a finite extension, we can assume that $G$ is simply connected as an algebraic group over $F$ (Lemma 5.5 in \cite{liao}). In that case, by Lemma 5.3 and 5.4 in \cite{liao}, we are in the same situation as in the real case and the rest of the proof applies with no change.

\begin{rem}\label{rem:proof_other_alggroups_Banach} The preceding argument shows more generally that if a higher rank \emph{simple} group $G$ contains a group locally isomorphic to a group with property $(*_{\mathcal E})$, then $G$ also satisfies $(*_{\mathcal E})$. Together with Remarks \ref{rem:proof_SL3F_Banach} (resp. \ref{rem:proof_SL3R_Banach} and Subsection \ref{subsection:SLbigN}) this proves Theorem \ref{thm:main_Banach_valued_nonarch} (resp. \ref{thm:main_Banach_valued_real}) for higher rank groups.
\end{rem}
\begin{rem}\label{rem:nonsemisimple} The above proof for simple Lie groups shows the following more general fact. Let $H$ be a closed subgroup of a connected Lie group $G$, with Lie algebras $\mathfrak{h} \subset \mathfrak{g}$. Assume that $H$ has (*) (respectively $(*_{\mathcal E})$), and that $\mathfrak g$ is the smallest Lie subalgebra containing $\mathfrak{h}$ and
  \[ \{Y \in \mathfrak{g} : \exists X \in \mathfrak{h}, [X,Y]=Y\}.\]
  Then $G$ has (*) (respectively $(*_{\mathcal E})$).

  For example, $\SL(n,\R) \ltimes \R^n$ has (*) for $n \geq 3$, and even $(*_{\mathcal E})$ for every class ${\mathcal E}$ for which $\SL(n,\R)$ has $(*_{\mathcal E})$ .
\end{rem}

\section{Passing to lattices}\label{sec:lattices}

\subsection{Facts on lattices}
We collect two facts on lattices.

The first one is a celebrated theorem from \cite{MR1828742} which asserts that the embedding of a lattice in a higher rank group is a bilipschitz map (we call such a lattice \emph{undistorted}).
\begin{thm}(Lubotzky--Mozes--Raghunathan) \label{thm:LMR} Let $G$ be a higher rank group with word-length $\ell_G$ with respect to some compact generating set, and $\Gamma \subset G$ be a lattice with word-length $\ell_\Gamma$ with respect to some finite generating set. There is $C>0$ such that for every $\gamma \in \Gamma$,
  \[ \ell_\Gamma(\gamma) \leq C \ell_G(\gamma).\]
  \end{thm}
The results in \cite{MR1828742} do not formally include the preceding statement, as they do not include the non-algebraic groups (as $\widetilde{\mathrm{Sp}}_{2n}(\R)$). However the general case follows, as in \cite[Section 7]{strongTsp4}, from the following Lemma which is certainly well-known.
\begin{lemma}\label{lem:image_of_lattice_lattice} Let $G$ be a higher rank group, with center $Z(G)$ and $\Gamma \subset G$ be a lattice. The image of $\Gamma$ in $G/Z(G)$ is a lattice in $G/Z(G)$.
\end{lemma}
\begin{proof} Write $G=\prod_{i \in I}G_i$ the decomposition of $G$ into finitely many simple pieces: each $G_i$ is either a connected simple Lie group of real rank $\geq 2$, or a connected almost $\F_i$-simple algebraic group of $\F_i$-split rank $\geq 2$ over a local field $\F_i$. 

$\Gamma$ preserves a probability measure on $G/Z(G)$, so we have to prove that the image of $\Gamma$ in $G/Z(G)$ is discrete, or equivalently that $\Gamma Z(G)$ is discrete in $G$. For this we prove that the centralizer $C$ of $\Gamma$ coincides with $Z(G)$, and in particular is discrete. This is enough: by property (T), $\Gamma$ is finitely generated, and so discreteness of its centralizer is equivalent to discreteness of its normalizer (which contains $\Gamma Z(G)$).

  For each $i \in I$, the group $G_i/Z(G_i)$ is a linear algebraic group over $F_i$, and the image in it of $C$ centralizes the Zariski closure of the image of $\Gamma$, which is $G_i/Z(G_i)$ by Theorem II.2.5 and Lemma II.2.3 in \cite{MR1090825}. Since $G_i/Z(G_i)$ is centerless, this implies that the image of $C$ in $G_i/Z(G_i)$ is trivial. This proves that $C \subset \prod_{i \in I} Z(G_i)=Z(G)$ and the Lemma.
  \end{proof}
\begin{proof}[Proof of Theorem \ref{thm:LMR}] Let $G$ be a higher rank group and $\Gamma \subset G$ a lattice. Let $G'=G/Z(G)$ and $\Gamma' \subset G'$ be the image of $\Gamma$ in the quotient. Denote by $\ell,\ell',\ell_\Gamma,\ell_{\Gamma'},\ell_Z$ the word-length functions on $G,G',\Gamma,\Gamma',Z(G)$ with respect to some compact generating sets.  By Lemma \ref{lem:image_of_lattice_lattice}, $\Gamma'$ is a lattice in $G'$, and so $\Gamma$ has finite index in the preimage in $G$ of $\Gamma$. Without loss of generality we can assume that $\Gamma$ is actually equal to the preimage in $G$ of $\Gamma'$.

$G'$ is a finite product of connected almost simple algebraic groups over local fields, so we can apply \cite{MR1828742} to every irreducible component of $\Gamma'$ and obtain that $\ell'_{\Gamma'}$ and $\ell'\left|_{\Gamma'}\right.$ are quasi-isometric.

  By \cite{MR510552} the central extension $G$ of $G'$ is given by a bounded $2$-cocycle. This implies that $G$ is quasi-isometric to $G' \times Z(G)$ and that $(\Gamma,\ell_\Gamma)$ is quasi-isometric to $(\Gamma' \times Z(G),\ell_{\Gamma'} \times \ell_{Z(G)})$ with compatible maps. Puting everything together we get
  \[ (\Gamma,\ell_\Gamma) \simeq (\Gamma' \times Z(G),\ell_{\Gamma'} \times \ell_{Z(G)}) \simeq (\Gamma' \times Z(G),\ell'\left|_{\Gamma'}\right. \times \ell_{Z(G)}) \simeq (\Gamma,\ell\left|_{\Gamma}\right.)\]
  which proves the Theorem.
  \end{proof}  

The second result says that, for higher rank lattices, the measure of
the cusps in $G/\Gamma$ decay exponentially fast. In a first version
of this paper, I had sketched a proof, similar to \cite[Section
  VII.1]{MR1090825}, relying on Margulis' arithmeticity theorem and
the Harish-Chandra--Borel--Behr--Harder reduction theorem for
$S$-arithmetic lattices. Fran\c{c}ois Maucourant explained to me that
this is a direct consequence of property (T) (actually even of spectral gap), and that it applies more generally to all Lie groups and all simple algebraic groups over local fields \cite{MR2600524,MR1415573,MR2782172}. I
thank him for allowing me to include this proof here.
\begin{thm}\label{thm:measure_of_cusps}(Maucourant)
  Let $G$ be a property (T) locally compact group with length function
  $\ell$, and $\Gamma \subset G$ be a lattice. There is a Borel
  fundamental domain $\Omega \subset G$ and $s_0>0$ such that
\[ \int_{\Omega} e^{s_0 \ell(g)} dx < \infty.\]
More generally, if $\Gamma \subset G$ is a lattice in a locally
compact group equipped with a length function $\ell$, and if and $G
\acts G/\Gamma$ has spectral gap, then there is a compact subset $Q
\subset G$ such that the image of $Q^n$ in
$G/\Gamma$ has measure $\geq 1- 2^{-n}$.
\end{thm}
\begin{proof}
The second statement is clearly more general than the first, so let us
focus on the second. Equip $G/\Gamma$ with the unique $G$-invariant
probability measure, and consider $\lambda$, the regular
representation of $G$ on $L_2(G/\Gamma)$. The invariant vectors are
exactly the constant functions on $G/\Gamma$. Denote by $P \colon
L_2(G/\Gamma) \to L_2(G/\Gamma)$ the orthogonal projection on the
constant functions, \emph{i.e.} the linear map sending $f$ to the
constant function equal to $\int_{G/\Gamma} f$. By Remark
\ref{rem:spectral_gap}, there is a symmetric compact subset $Q \subset
G$ and a sequence of probability measures on $Q^n$ such that
$\|\lambda(\mu_n) - P\| \leq 2^{-n}$. We may assume that $Q$ has
positive measure.

For an integer $n$, denote by $\Omega_n$ the image of $Q^n$ in
$G/\Gamma$. Let $f_n$ the indicator function of $\Omega_n$. If $g \in
Q^{n}$, we have that $g Q^{n+1}$ contains $Q$, so
  \[ \langle \lambda(g) f_{n+1},f_1\rangle = | g \Omega_{n+1} \cap \Omega_1| = |\Omega_1|\]
and
\[ \langle \lambda(\mu_n) f_{n+1},f_1\rangle  = |\Omega_1|\]
On the other hand, we have
\[ \langle P f_{n+1},f_1\rangle = |\Omega_{n+1}| |\Omega_1|.\]
So we have
\[ |\Omega_1| (1-|\Omega_{n+1}|) = \langle (\lambda(\mu_n) - P) f_{n+1},f_1\rangle \leq 2^{-n}.\]
We deduce $1-|\Omega_{n+1}| \leq C 2^{-n} = 2C 2^{-n-1}$ for $C =
\frac{1}{|\Omega_1|}$, which is finite because $Q$ has positive
measure. Replacing $Q$ by some power allows to remove the factor $2C$.
\end{proof}

\subsection{Inducing from exponentially integrable lattices}\label{subsection:induction}
Our goal in this subsection is to prove the following result.
\begin{thm}\label{thm:exponentially_integrable_lattices} Let $\Gamma \subset G$ be a lattice in a locally compact group $G$ and $\ell$ be a length function on $G$. Then $(G,\ell)$ has (*) if and only if $(\Gamma,\ell\left|_{\Gamma}\right.)$ has (*).
\end{thm}
\begin{rem}\label{rem:exponentially_integrable_lattices_Banach} More generally, if $\mathcal E$ is a class of Banach spaces which is stable by $X\mapsto L_2(\Omega,\mu;X)$ for every measure space $(\Omega,\mu)$, then $(G,\ell)$ has $(*_{\mathcal E})$ iff $(\Gamma,\ell\left|_{\Gamma}\right)$ has.\end{rem}
\begin{proof}
We can assume that $G$ has property (T), as otherwise $\Gamma$ does not have property (T) either and the theorem is empty. By Theorem \ref{thm:measure_of_cusps}, $\Gamma$ admits a fundamental domain $\Omega \subset G$ and $s_0>0$ such that
  \begin{equation}\label{eq:hyp_exp_integrable} \int_{\Omega} e^{s_0 \ell(x)} dx<\infty.\end{equation}

  We can assume furthermore that every element $\omega$ of $\Omega$ almost minimizes the length of its $\Gamma$-orbit, for example that it satisfies
  \begin{equation}\label{eq:ell_almost_min_on_Omega}\ell(\omega) \leq \inf_{\gamma \in \Gamma} \ell(\omega \gamma)+1\textrm{ for every }\gamma \in \Gamma.\end{equation} Indeed, if $(\gamma_n)_{n \geq 0}$ is an enumeration of $\Gamma$ (say with $\gamma_0=e$) and if $f(\omega) = \omega \gamma_n$ for the first $n$ satisfying $\ell(\omega \gamma_n)  \leq 1+\inf_\gamma \ell(\omega \gamma)$, then we can replace $\Omega$ by $f(\Omega)$, which remains a Borel fundamental domain for $\Gamma$, which still satisfies \eqref{eq:hyp_exp_integrable} and which moreover satisfies \eqref{eq:ell_almost_min_on_Omega}. Of course, in most interesting cases, $\ell$ is proper and we can remove the $+1$ in \eqref{eq:ell_almost_min_on_Omega}.

For convenience we choose the normalization of the Haar measure on $G$ so that $\Omega$ has measure one.

For every $g \in G$ and $\omega \in \Omega$ we denote by \[g\omega=(g \cdot \omega) \alpha(g,\omega)\] the unique decomposition of $g\omega$ as a product of $g\cdot \omega \in \Omega$ and $\alpha(g,\omega) \in \Gamma$. Recall that $(g,\omega) \mapsto g \cdot \omega$ is a probability measure preserving action of $G$ on $\Omega$, and that $\alpha$ is a cocycle, \emph{i.e.} is satisfies the \emph{cocycle relation}
\[ \alpha(g_1g_1,\omega) =\alpha(g_1,g_2\cdot \omega)\alpha(g_2,\omega)\]
for every $g_1,g_2\in G$ and $\omega \in \Omega$. We start with a Lemma which shows that, in a sense close to Shalom's notion of $L_p$-integrable lattice, \eqref{eq:hyp_exp_integrable} implies that $\Gamma$ is \emph{exponentially integrable}.
\begin{lemma}\label{lem:exp_int_cocycle} There is $C>0$ such that for every $s\leq\frac{s_0}{2}$,
  \[ \int_\Omega e^{s\ell(\alpha(g,\omega))} d\omega \leq C e^{2s\ell(g)}.\]
\end{lemma}
\begin{proof} We write $\alpha(g,\omega) = (g\cdot \omega)^{-1}g\omega$. By the symmetry of $\ell$ and \eqref{eq:ell_almost_min_on_Omega}, we have
  \[ \ell((g\cdot \omega)^{-1} ) = \ell(g\cdot \omega) \leq \ell(g\omega)+1.\]
  By the subadditivity of $\ell$ we deduce
  \[ \ell(\alpha(g,\omega)) \leq 1+2 \ell(g\omega) \leq 1+2\ell(g) + 2\ell(\omega).\] The lemma follows by integrating, with $C= e^{s_0/2}\int_\Omega e^{s_0\ell(\omega)} d\omega.$
  \end{proof}

We start with the interesting direction, namely the implication $(G,\ell)$ has (*)$\implies$ $(\Gamma,\ell\left|_{\Gamma}\right.)$ has (*).

Let $\widetilde s,\widetilde t,\widetilde C>0$ and $\widetilde m_n$ be a sequence of probability measures on $\{g \in G|\ell(g) \leq n\}$ as in (*) for $(G,\ell)$. 

Let $\pi \colon \Gamma \to B(X_0,X_2)$ be such that $\pi(gg')$ factors as $\pi_1(g) \pi_0(g')$ for a Hilbert space $X_1$ and two maps $\pi_0\colon \Gamma \to B(X_0,X_1)$ and $\pi_1\colon \Gamma \to B(X_1,X_2)$. Assume that there is $L>0$ such that  $\|\pi_i(\gamma)\|\leq L e^{s \ell(g)}$ for all $\gamma \in \Gamma$ and $i \in \{0,1\}$, with $s=\min(\frac{s_0}{4},\frac{\widetilde s}{2})$.

For $x \in X_0$, $g \in G$ and $\omega \in \Omega$, define $\widetilde \pi_0(g)x (\omega) = \pi_0(\alpha(g^{-1},\omega)^{-1}) x \in X_1$. By Lemma \ref{lem:exp_int_cocycle}, since $s\leq\frac{s_0}{4}$ the measurable function $\widetilde \pi_0(g)x$ is $L_2$-integrable:
\[ \|\widetilde \pi_0(g)x\|_{L_2(\Omega;X_1)} \leq \left(\int_\Omega  L^2 e^{2s \ell(\alpha(g^{-1},\omega))} \|x\|^2d\omega\right)^{\frac 1 2} \leq C L e^{2s \ell(g)}.\]
This proves that $\widetilde \pi_0(g)$ has norm $\leq CL e^{2s\ell(g)}$ from $X_0$ to $L_2(\Omega;X_1)$.
Similarly, for every $f \in L_2(\Omega;X_1)$ we can bound
\begin{eqnarray*} \int \|\pi_1(\alpha(g^{-1},\omega)^{-1} f(g^{-1}\cdot \omega)\|_{X_2} d\omega & \leq & \int L e^{s\ell(\alpha(g^{-1},\omega))} \|f(g^{-1}\cdot \omega)\|_{X_1} d\omega\\
  & \leq & L \left (\int e^{2s \ell(\alpha(g^{-1},\omega))} d\omega\right)^{\frac 1 2} \|f\|_2.\end{eqnarray*}
Therefore, we can define a map $\widetilde \pi_1(g) \colon L_2(\Omega;X_1) \to X_2$ by \[\widetilde \pi_1(g) f = \int \pi_1(\alpha(g^{-1},\omega)^{-1}) f(g^{-1}\cdot \omega) d\omega=\int \pi_1(\alpha(g,\omega)) f(\omega) d\omega,\] and it satisfies $\|\widetilde \pi_1(g) \| \leq CLe^{2s\ell(g)}$. It is easy to check that $\widetilde \pi_0$ and $\widetilde \pi_1$ are strongly continuous.
By the cocycle formula, we have
\[\widetilde \pi_1(g) \widetilde \pi_0(g')= \widetilde \pi(gg')\]
where
\[\widetilde \pi(g) x = \int \pi(\alpha(g^{-1},\omega)^{-1}) x d\omega.\]
So since $s \leq \frac{\widetilde s}{2}$, the map $\widetilde \pi$ satisfies the assumption in (*), so that there is $P \in B(X_0,X_2)$ such that
\[ \| \widetilde \pi(\widetilde m_n) - P \| \leq \widetilde C C^2 L^2 e^{-\widetilde{t}n}\]
and for every $g_1,g_2 \in G$
\[ \lim_n \|\widetilde \pi(\widetilde m_n) - \widetilde \pi(\delta_{g_1}\widetilde m_n \delta_{g_2}) \| =0.\]
Let us denote by $m_n^{(0)}$ the probability measure on $\Gamma$ given as the image of $m_n \otimes d\omega$ by the map $(g,\omega)\mapsto \alpha(g^{-1},\omega)^{-1}$. Then by Fubini's theorem one can write $\widetilde \pi(\widetilde m_n) = \pi(m_n^{(0)})$. To summarize, we have proven that whenever $\pi \colon \Gamma \to B(X_0,X_1)$ is as above with $s = \min(\frac{s_0}{4},\frac{\widetilde s}{2})$, then $\pi(m_n^{(0)})$ is well-defined (\emph{i.e.} the series $\sum_\gamma m_n^{(0)}(\gamma) \pi(\gamma)$ converges in norm) and is Cauchy in $B(X_0,X_2)$. In particular, for every $\gamma_1,\gamma_2 \in \Gamma$, $\lim_n \pi(\gamma_1 m_n^{(0)} \gamma_2)$ exists.

We shall prove that this limit is $P$. To do so we will prove that
\begin{equation}\label{eq:goal_P_Gamma_invariant} \lim_n \pi(\gamma_1 m_n^{(0)})x =Px\textrm{  for every }x \in X_0.\end{equation}
for every $\gamma_1 \in \Gamma$. Before we do so, let us explain how \eqref{eq:goal_P_Gamma_invariant} allows us to conclude the proof of Theorem \ref{thm:exponentially_integrable_lattices}. By the preceding discussion, \eqref{eq:goal_P_Gamma_invariant} implies the seemingly stronger conclusion that $\lim_n \|\pi(\gamma_1 m_n^{(0)}) -P\|=0$. By symmetry (consider the dual maps $\pi(g^{-1})^* \in B(X_2^*,X_0^*)$) \eqref{eq:goal_P_Gamma_invariant} implies that 
  \[ \lim_n \|\pi(m_n^{(0)}\gamma_2) - P \| =0\]
  for every $\gamma_2 \in \Gamma$. And hence $\|\pi(\gamma_1 m_n^{(0)} \gamma_2) - \pi(m_n^{(0)}) \|$, which is less than
\[ \|\pi(\gamma_1 m_n^{(0)} \gamma_2) - \pi(\gamma_1 m_n^{(0)}) \| + \|\pi(\gamma_1 m_n^{(0)}) - \pi(m_n^{(0)}) \|\]
by the triangle inequality, goes to zero for every $\gamma_1,\gamma_2 \in \Gamma$ (for the first term, this is \eqref{eq:goal_P_Gamma_invariant} applied to the map $\gamma \mapsto \pi(\gamma_1 \gamma)$). The preceding inequality is almost (*) for $(\Gamma,\ell)$, except that $m_n^{(0)}$ is not supported on $B_n := \{\gamma \in \Gamma | \ell(\gamma) \leq n\}$. This is however almost true for $m_{n/5}$, as will be deduced from the following lemma.
\begin{lemma}\label{lem:tail_of_mn0} If $s \leq \frac{s_0}{4}$, then for every integer $n$
  \[ \int e^{s\ell(\gamma)} 1_{\ell(\gamma) \geq 5n} dm_n^{(0)}(\gamma) \leq C e^{(-\frac{3s_0}{2}+5s)n} \leq Ce^{-\frac{s_0}{4}n}.\]
  \end{lemma}
\begin{proof} We have $e^{s\ell(\gamma)} 1_{\ell(\gamma) \geq 5n} \leq e^{\frac{s_0}{2} \ell(\gamma) - 5(\frac{s_0}{2}-s)n}$. By integrating and using the definition of $m_n^{(0)}$ we obtain
  \[ \int e^{s\ell(\gamma)} 1_{\ell(\gamma) \geq 5n} dm_n^{(0)}(\gamma) \leq \sup_{\ell(g) \leq n} \int e^{\frac{s_0}{2} \ell(\alpha(g^{-1},\omega)) - 5(\frac{s_0}{2}-s)n} d\omega.\]
  One concludes by Lemma \ref{lem:exp_int_cocycle}.
\end{proof}
So one defines $m_n$ as the conditional probability $m_{n/5}^{(0)}(\cdot \cap B_n)/m_{n/5}^{(0)}(B_n)$. By definition $m_n$ is supported on $B_n$. On the other hand, Lemma \ref{lem:tail_of_mn0} provides a constant $C'$ (depending on $C$ and $s_0$ only) such that, if $s \leq \frac{s_0}{4}$,
\[\|\pi(m_n) - \pi(m_{n/5}^{(0)})\| \leq C' L^2 e^{-\frac{s_0}{20}n}.\]
Therefore
\[\|\pi(m_n) - P \| \leq C'' L^2 e^{-t n}\]
for $t=\min(\frac{s_0}{20},\frac{\widetilde t}{5})$ and $C''=C'+\widetilde C C^2$. Similarly
\[\lim_n \|\pi(\gamma_1 m_n \gamma_2) - \pi(\gamma_1 m_{n/5}^{(0)}\gamma_2)\|=0\]
for every $\gamma_1,\gamma_2 \in \Gamma$. This proves (*) for $(\Gamma,\ell)$ with $s = \min(\frac{s_0}{4},\frac{\widetilde s}{2})$, $t=\min(\frac{s_0}{20},\frac{\widetilde t}{5})$.

It remains to justify \eqref{eq:goal_P_Gamma_invariant}.

Fix a probability measure $\nu_0$ on $\Gamma$ with full support and satisfying $\int e^{s\ell(\gamma)} d\nu_0(\gamma)<\infty$. Let $\nu$ be the probability measure on $G$ given by $\int f d\nu = \int f(\omega \gamma) d\omega d\nu_0(\gamma)$. We define a new map $\widetilde \pi_1'(g) \colon L_2(\Omega;X_1) \to L^1(G,\nu;X_2)$ by setting
\[\widetilde \pi_1'(g) f(\omega\gamma) = \pi_1(\gamma^{-1} \alpha(g^{-1},\omega)^{-1}) f(g^{-1}\cdot \omega),\]
so that $\widetilde \pi_1(g) f= \int (\widetilde \pi'_1(g) f)(\omega) d\omega$. To check that $\widetilde \pi'_1(g)$ maps $L_2(\Omega;X_1)$ to $L^1(G,\nu;X_2)$, we compute
\begin{multline*}\int \| \pi_1(\gamma^{-1} \alpha(g^{-1},\omega)^{-1}) f(g^{-1}\cdot \omega) \| d\omega d\nu_0(\gamma) \\ \leq L \int e^{s\ell(\gamma)} e^{s\ell(\alpha(g^{-1},\omega))} \|f(g^{-1}\cdot \omega\| d\omega d\nu_0(\gamma).\end{multline*}
This is less than
\[ L \int e^{s\ell(\gamma)} d\nu_0(\gamma) C e^{2s\ell(g)} \|f\|_2\]
by the Cauchy-Schwarz inequality and Lemma \ref{lem:exp_int_cocycle}. Also, we have $\widetilde \pi'_1(g) \widetilde \pi_0(g') = \widetilde \pi'(gg')$ where \[\widetilde \pi'(g)x(\omega\gamma) = \pi(\gamma^{-1} \alpha(g^{-1},\omega)^{-1})x\] for $x \in X_0$. So (recall $2s \leq \widetilde s$) we can apply (*) for $(G,\ell)$. In particular, if $x \in X_0$ we get that $f_n=\widetilde \pi'(\widetilde m_n)x$ converges to some $f \in L^1(G,\nu;X_2)$, and $\widetilde \pi'(g \widetilde m_n)x$ converges to the same $f$ for every $g \in G$. 

On the one hand, for $g,g'\in G$ and almost every $\omega \gamma \in G$, one checks from the definitions that 
\[ (\widetilde\pi'(gg')x)(\omega \gamma) = (\widetilde \pi'(g')x)(g^{-1}\omega \gamma).\]
This means that both functions $f_n$ and $h\mapsto f_n(g^{-1} h)$ converge in $L^1(G,\nu;X_2)$ to $f$. Since $\nu$ is equivalent to the Haar measure of $G$, this implies that $f(g^{-1}h) = f(h)$ for almost every $h \in G$. But this holds for every $g \in G$, therefore there exists $y \in X_2$ such that $f(h) = y$ for $\nu$-almost every $h$.

On the other hand, by definition of $m_n^{(0)}$ and by the fact that $\widetilde \pi(g)$ is the composition of $\widetilde \pi'(g)$ and of the map $f \mapsto \int f d\nu$, we have
\[ \pi(m_n^{(0)})x= \widetilde \pi(\widetilde m_n)x = \int f_n(\omega) d\omega\]
and $Px = \int f(\omega) d\omega = y$.
Similarly,
\begin{eqnarray*} \pi(\gamma_1 m_n^{(0)})x&=& \iint \pi(\gamma_1 \alpha(g^{-1},\omega)^{-1})x d\omega d\widetilde m_n(g) \\
  &=&\int f_n(\omega \gamma_1^{-1}) d\omega.
\end{eqnarray*}
By our choice of $\nu_0$, the map $h \mapsto \int h(\omega \gamma_1^{-1}) d\omega$ is bounded on $L^1(G,\nu;X_2)$. This implies that $\lim_n \pi(\gamma_1 m_n^{(0)}) x=\int f(\omega \gamma_1^{-1})  d\omega$, which is equal to $y=Px$. This concludes the proof of \eqref{eq:goal_P_Gamma_invariant} and therefore of the \emph{only if} direction in Theorem \ref{thm:exponentially_integrable_lattices}.

The if direction is easier. Assume that $(\Gamma,\ell\left|_\Gamma\right.)$ satisfies (*), with $s,t,C,(m_n)_{n \geq 0}$. Without loss of generality we can assume that $s \leq s_0$. Let $\pi\colon G \to B(X_0,X_1)$ be as in the definition of (*) with this value of $s$. We can apply (*) to the restriction of $\pi$ to $\Gamma$. More generally, for every measure $\mu_1,\mu_2$ on $G$ such that $C(\mu_i) := L \int e^{s \ell(g)} d\mu_i(g) <\infty$, we can apply (*) to the map $\gamma \mapsto \pi(\mu_1 \delta_{\gamma} \mu_2)$ and get an operator ${}_{\mu_1}P_{\mu_2}$ in $B(X_0,X_1)$ satisfying
\[ \| {}_{\mu_1}P_{\mu_2} - \pi(\mu_1 m_n \mu_2)\| \leq C C(\mu_1)C(\mu_2) e^{-tn}\]
and ${}_{\mu_1\delta_{\gamma_1}}P_{\delta_{\gamma_2}\mu_2}={}_{\mu_1}P_{\mu_2}$ for every $\gamma_1,\gamma_2 \in \Gamma$. But since $s \leq s_0$, we can in particular apply the preceding to $\mu_1=\mu_2=\mu$ the probability measure on $\Omega$, and to its translates by any $g \in G$. For $g_1 \in G$, we can decompose the probability measure $\delta_{g_1} \mu$ as a sum $\sum_{\gamma \in \Gamma} \mu_\gamma \delta_{\gamma}$ for a family of measures $\mu_\gamma$ on $\Omega$ suming to $\mu$. This leads to the equality
\[ {}_{g_1\mu}P_{\mu} = \sum_{\gamma} {}_{\mu_\gamma \delta_{\gamma}}P_{\mu} = \sum_{\gamma} {}_{\mu_\gamma}P_{\mu} = {}_{\mu}P_{\mu}\]
where the easy justifications of the convergence are left to the reader. Similarly ${}_{g_1 \mu}P_{\mu g_2} = {}_{\mu}P_{\mu}$ for every $g_1,g_2 \in G$. This is almost (*) for $G$ and the sequence of probability measures $\widetilde m_n = \mu \ast m_n \ast \mu$. The only issue is that $\widetilde m_n$ is not supported in $\{g \in G|\ell(g) \leq n\}$. This is fixed by suitably truncating $\widetilde m_n$ as in the only if direction.
\end{proof}
We also have the following variant of the easy direction in the preceding theorem.
\begin{prop}\label{prop:StrongTafterrestriction} Let $G,\Gamma,\ell$ be as in Theorem \ref{thm:exponentially_integrable_lattices}. There exists $s_0>0$ such that, for every $s \leq s_0$ and $c\geq 0$, if $\mathcal C_{s\ell+c}(\Gamma)$ has a Kazhdan projection, then so does $\mathcal C_{s\ell+c}(G)$.

More generally if $G \acts G/\Gamma$ has spectral gap, if $\mathcal E$ is a class of Banach spaces and if $\mathcal C_{s\ell+c,\mathcal E}(\Gamma)$ has a Kazhdan projection, then so does $\mathcal C_{s\ell+c,\mathcal E}(G)$.
\end{prop}
\begin{proof} If $\mathcal C_{s\ell+c}(\Gamma)$ has a Kazhdan projection for some $s,c \geq 0$, then in particular $\Gamma$ has property (T) and $G$ also. Let $\Omega$ be the fundamental domain et $s_0>0$ the real number given by Theorem \ref{thm:measure_of_cusps}. Let $\mu$ be the uniform probability measure on $\Omega$. If $s \leq s_0$, $\mu$ belongs to $C_{s\ell+c,\mathcal E}(G)$ with norm $\leq \int_\Omega e^{s\ell(g) +c} d\mu(g)$, and more generally the map $f \in \C[\Gamma] \mapsto \mu f $ (the convolution of $\mu$ and of $f$, seen as the measure $\sum_\gamma f(\gamma) \delta_\gamma$ on $G$) extends to a linear map of norm $\leq  \int_\Omega e^{s\ell(g) +c}d\mu(g)$ from $\mathcal C_{s\ell+c,\mathcal E}(\Gamma)$ to $\mathcal C_{s\ell+c,\mathcal E}(G)$. We claim that if $P \in \mathcal C_{s\ell+c,\mathcal E}(\Gamma)$ is a Kazhdan projection, then $\mu P \in \mathcal C_{s\ell+c,\mathcal E}(G)$ is also a Kazhdan projection. We have to prove that for every $g_1 \in G$, $\delta_{g_1} \mu P = \mu P$. As in the preceding proof, we can decompose the probability measure $\delta_{g_1} \mu$ as a sum $\sum_{\gamma \in \Gamma} \mu_\gamma \delta_{\gamma}$ for a family of measures $\mu_\gamma$ on $\Omega$ suming to $\mu$. This leads to the desired formula
  \[ \delta_{g_1} \mu P =  \sum_{\gamma} \mu_\gamma \delta_{\gamma}P= \sum_{\gamma} \mu_\gamma P = \mu P\]
  where the middle inequality is because $P$ is a Kazhdan projection for $\Gamma$, and where the justification of the summability of both series is straighforward. 
\end{proof}
A direct corollary of Proposition \ref{prop:StrongTafterrestriction} is that the negation of strong property (T) is inherited by lattices in full generality.
\begin{cor}\label{cor:NotStrTpasses_to_lattices} Let $\Gamma$ be a lattice in a locally compact group $G$. If $G$ does not have strong property (T), then $\Gamma$ does not have strong property (T) either.
\end{cor} 

\subsection{End of proof of Theorem \ref{thm:main}}
The fact that every higher rank group satisfies property (*) has already been proven in Section \ref{sec:all_higher_rank_groups}. It remains to prove it for a lattice $\Gamma$ in higher rank group $G$. Let $\ell$ be the word-length function on $G$ with respect to some compact generating set. By Lemma \ref{lem:changing_length} and Lubotzky--Mozes--Raghunathan's Theorem \ref{thm:LMR}, it is enough to prove that $(\Gamma,\ell\left|_\Gamma\right.)$ satisfies (*). By Theorem \ref{thm:measure_of_cusps} and Theorem \ref{thm:exponentially_integrable_lattices}, this follows from the fact, already proven, that $G$ satisfies (*).

\begin{rem}\label{rem:proof_lattices} If we take into account Remark \ref{rem:exponentially_integrable_lattices_Banach} and note that having nontrivial Rademacher type, as well as \eqref{eq:Stheta}, \eqref{eq:Tdelta}, \eqref{eq:Tdeltan} are all Banach-space properties which are stable by the operation $X \mapsto L_2(\Omega,\mu;X)$, we complete similarly the proofs of Theorem \ref{thm:main_Banach_valued_nonarch} and Theorem \ref{thm:main_Banach_valued_real}. In fact we get the following more general result. In the statement, if $G = \prod_{i \in I} G_i$ is a product of higher rank simple groups, the real factors are those $G_i$'s which are real Lie groups, whereas the non-archimedean factors are the others, that is those $G_i$'s which are algebraic groups over non-archimedean local fields.
\end{rem}

\begin{thm}\label{thm:main_Banach_general} Let $G$ be a higher rank group,  $\Gamma \subset G$ a lattice and $\mathcal E$ a class of Banach spaces. Then both $G$ and $\Gamma$ have $(*_{\mathcal E})$ (and therefore strong (T) with respect to $\mathcal E$) if one of the following conditions holds:
  \begin{itemize}
  \item $G$ has no real factor and $\mathcal E$ has nontrivial Rademacher type.
  \item the Lie algebra of every real factor of $G$ contains a Lie subalgebra isomorphic to $\mathfrak{sp}_4$ or $\mathrm{sl}_3$, and there is $\alpha\in (0,1]$ and $C>0$ such that \eqref{eq:Stheta} and \eqref{eq:Tdelta} hold for every $X \in \mathcal E$.
  \item the Lie algebra of every real factor of $G$ contains a Lie subalgebra isomorphic to $\mathfrak{sl}_{3n-3}$ for $n \geq 2$, and there is $\alpha\in (0,1]$ and $C>0$ such that \eqref{eq:Tdeltan} for every $X \in \mathcal E$.
  \end{itemize}
\end{thm}

\bibliographystyle{plain_correctalpha} 


\end{document}